\documentclass[preprint,12pt]{elsarticle}
\usepackage{amssymb}
\usepackage{amsmath}
\usepackage{amsthm}
\usepackage{mathtools}
\usepackage{bm}
\usepackage{multirow}
\usepackage{setspace}
\usepackage{subcaption}
\usepackage{caption}
\usepackage{tabularx}
\usepackage{booktabs}
\usepackage{enumitem}
\usepackage{natmove}
\biboptions{sort&compress}
\usepackage[hidelinks]{hyperref}
\usepackage{graphicx}
\usepackage{algorithm}
\usepackage{algpseudocode}
\usepackage{xcolor}
\usepackage{float}

\newtheorem{assumption}{Assumption}

\newtheorem{remark}{Remark}
\newtheorem{lemma}{Lemma}
\newtheorem{theorem}{Theorem}

\usepackage{orcidlink}

\setlist{nosep,leftmargin=*}
\journal{Journal of Process Control}

\begin{document}

\begin{frontmatter}

\title{Real-Time Nonlinear Model Predictive Control Framework for Event-Triggered Switching in Industrial Batch Polymerization Process}

\author[1]{Chenchen Zhou\orcidlink{0000-0003-2993-8353}}
\ead{chenchen.zhou@kuleuven.be}
\author[2]{Zuzhen Ji\orcidlink{0000-0002-6213-961X}}
\ead{jizuzhen@zjut.edu.cn}
\author[1]{Jose Matias\orcidlink{0000-0002-1094-6145}\corref{cor1}}
\ead{jose.matias@kuleuven.be}

\cortext[cor1]{Corresponding author.}
\address[1]{Chemical and Biochemical Reactor Engineering and Safety (CREaS), KU Leuven, Belgium}
\address[2]{Department of Mechanical Engineering, Zhejiang University of Technology, Hangzhou, Zhejiang Province, P.R. China}

\begin{abstract}
Controlling batch polymerization is challenging because the absence of a steady operating point prevents standard linearization; the dynamics are intrinsically nonlinear; and multi-phase operation induces state-triggered switching. This study systematically combines four established real-time NMPC ingredients, smooth mode blending, advanced-step warm starts, variable scaling, and a capped iteration budget, to attain real-time feasibility without ad hoc switching heuristics. We provide practice-oriented guidance for selecting smoothing gains and locating switching surfaces, and we make explicit the approximations introduced by smoothing such that, with appropriate tuning, the smoothed and original switching logic are numerically indistinguishable at solver-tolerance levels. All results are obtained in closed-loop simulation using an industrial gas--liquid polymerization benchmark with estimator-in-the-loop, compared against PID and conventional NMPC baselines. Results show improved constraint satisfaction and shorter batch duration under bounded computation, while an ablation study quantifies the specific contributions of each component individually.
\end{abstract}

\begin{highlights}
	\item Smooth mode blending removes ad hoc phase logic in batch NMPC.
	\item Tuning links smoothing gains and switch planes to sub-optimal NMPC.
	\item Real-time workflow: advanced-step, scaling, capped iterations.
	\item Industrial, estimator-in-the-loop study outperforms PID and conventional NMPC.
\end{highlights}


\end{frontmatter}

\section{Introduction}\label{sec:intro}
Batch polymerization plants have progressed rapidly to meet demand for high-performance polymers with tailored properties \cite{soroush1992}. Unlike steady-state continuous trains, batch operations deliver bespoke molecular architectures while remaining accountable for batch-to-batch reproducibility and safety obligations imposed by hazardous monomers and catalysts \cite{lu2017batch}. This combination of flexibility and regulatory rigor is a challenge for conventional control paradigms.

Three coupled features distinguish industrial batch polymerization from continuous processes. (i) Event-driven, multi-stage recipes trigger transitions via conversion thresholds or material additions, injecting disturbances that condition the remainder of the batch \cite{lu2017batch}. (ii) Reaction kinetics are strongly exothermic and nonlinear, so precise thermal management is needed to preserve quality and reactor safety. (iii) Simultaneous gas, liquid, and occasional solid phases introduce state-dependent transport phenomena that complicate modeling and feedback design \cite{harmonray2000a}.

These characteristics limit the use of linear controllers, which although simple, demand frequent retuning or manual switching between phases, adding operator burden and decreasing economic performance \cite{soroush1992,richards2006}. Nonlinear Model Predictive Control (NMPC) therefore attracts interest because it coordinates nonlinear dynamics, hard constraints, and economic objectives within a single optimization problem \cite{bindlish2015a,grune2017}.

\begin{table}[t]
\centering
\caption{Mode-dependent objectives and constraints in batch polymerization.}
\label{tab:poly_modes}
\begin{tabularx}{0.8\columnwidth}{cXXX}
\toprule
\textbf{Mode} & \textbf{Phase Description} & \textbf{Primary Objective} & \textbf{Key Constraints} \\
\midrule
1 & Heating and monomer addition & Minimize heating time & Temperature and pressure limits \\
2 & Main polymerization & Maximize conversion rate & Safety and quality constraints \\
3 & Cooling and product finishing & Achieve molecular weight targets & Cooling rate and final specifications \\
\bottomrule
\end{tabularx}
\end{table}

Industrial deployment nevertheless remains sparse because two obstacles dominate:
\begin{enumerate}
	\item \textbf{Event-driven multi-mode dynamics:} Batch polymerization involves state-triggered mode changes where each phase carries distinct objectives, constraints, and safety priorities, such as time minimization during heating, conversion management during polymerization, and molecular-weight regulation during cooling (see Table~\ref{tab:poly_modes}). Unlike fixed hierarchical schedules \cite{mhaskar2005,lucia2013a}, the switching logic depends on process states (e.g., conversion, temperature) rather than time, requiring NMPC formulations that anticipate multiple switching sequences \cite{zhang2016,ma2010}.

	\item \textbf{Real-time computational tractability:} Large, multi-phase reactor models generate high-dimensional nonlinear programs that become mixed-integer when switching logic is explicit \cite{sahinidis2019,zhou2024b}. Solving these problems within sampling periods of seconds-to-minutes pushes beyond available compute capacity for industrial deployments.
\end{enumerate}

In the literature, two main strategies are proposed to deal with those obstacles.Model simplification technologies, ranging from reduction \cite{nogueira2020a,prasad2002,gao2013} and objective approximation \cite{bindlish2003} to subsystem isolation that neglects coupled heat transfer \cite{park2003,joy2019}, have been widely used to deal with these obstacles. Despite the reduction in computational burden, these methods decrease the robustness of the control applications in cases where quality requirements tighten.
In a different context, advances in solvers for nonlinear programming as well as in hardware (such as IPOPT \cite{biegler2009} and CasADi \cite{andersson2019}) allow for the application of real-time NMPC in automotive \cite{wang2010,chen2024} and distributed-parameter applications \cite{zhang2023a,yang2021a}. Even with the availability of these high performance solvers, applications combining the event-driven switching with large-scale nonlinear dynamics are still limited in the literature, particularly in the field of batch polymerization. The main reason is the combinatorial complexity and non-differentiability of the switching logic that arise where the switching decisions are formulated using binary variables. 

Broader control research offers additional ingredients to overcome these obstacles, for example: hierarchical architectures decouple switching from continuous control \cite{mhaskar2005,lucia2013a}, multi-controller designs dedicate regulators to each mode \cite{zhang2016,ma2010}, and smoothing or advanced optimization techniques attempt to regularize hybrid dynamics \cite{mcallister2022,trespalacios2014,andrikopoulos2013,zhang2016a}. However, scalable performance on industrial batch polymerization problems remains difficult. Particularly for the event-driven switching NMPC formulation\cite{quarshie2025}, existing remedies fall into three categories: mixed-integer formulations that add binary variables for each mode \cite{rawlings2020,mcallister2022}, multiple-model strategies that solve distinct NMPC problems per mode \cite{bemporad2009,magni2008}, and heuristic switching logic that substitutes simplified rules for optimization \cite{corona2019}. None, however, have demonstrated scalable performance on industrial batch polymerization problems.

To address the challenges of industrial batch polymerization control, our approach focus on the systematic combination of existing methodologies. Variable scaling and iteration capping, which are standard good practices \cite{biegler2009}, are used to improve solver performance. Advanced-step NMPC is adapted from Zavala and Biegler~\cite{zavala2009a}, who developed the sensitivity-based online correction for general NLPs; we apply it unchanged but show its synergy with smoothing. Smooth approximations for discrete switches follow the logistic relaxation employed in \cite{mcallister2022,trespalacios2014}; our contribution is to quantify approximation error bounds (Section~\ref{sec:theoretical_analysis}) and provide tuning guidelines specific to batch polymerization.


In summary, this paper develops a real-time NMPC framework tailored to event-triggered switching in industrial batch polymerization by (i) capturing all operational phases within a unified control problem, (ii) defining switching through state-driven triggers instead of fixed schedules, and (iii) enforcing smooth control evolution across phase boundaries.

The contributions of this work are as follows:
\begin{enumerate}
	\item[(C1)] \textbf{Theoretical Analysis:} A smoothed NMPC switching formulation using logistic relaxations with rigorous bounds on approximation error and feasibility inclusions, characterizing the method as an $\epsilon$-suboptimal approximation scheme.
	\item[(C2)] \textbf{Practical Tuning:} A systematic procedure for specifying smoothing parameter values that balances approximation accuracy against numerical stability, avoiding common pitfalls like gradient vanishing or explosion.
	\item[(C3)] \textbf{Integrated Real-Time Framework:} A cohesive NMPC design enabling real-time operation by integrating the smoothed formulation with advanced-step warm-starting, variable scaling, and fixed iteration budgets.
	\item[(C4)] \textbf{Industrial Validation:} Quantitative validation on a complex, multi-phase polymerization reactor, demonstrating superior disturbance rejection and constraint handling compared to industry-standard PID.
\end{enumerate}

\textbf{Paper Organization:} Section~\ref{sec:cpf} introduces the event-driven modeling and unified NMPC structure. Section~\ref{sec:rt} details the real-time implementation with smoothing, scaling, and advanced-step warm starting. Section~\ref{sec:theoretical_analysis} summarizes theoretical guarantees and tuning guidance, Section~\ref{sec:case} presents industrial case studies, and Section~\ref{sec:con} outlines future research directions.

\section{Control Problem Formulation}
\label{sec:cpf}
\subsection{Batch Polymerization as a Switched System}
Industrial batch polymerization processes exhibit multi-mode behavior because operational phases are activated by state conditions rather than fixed schedules. Table~\ref{tab:poly_modes} summarizes the resulting phases, each with distinct objectives, constraints, and safety priorities determined by conversion thresholds, temperature limits, or quality targets.

The hybrid character of these processes breaks the single-mode assumption underlying conventional NMPC formulations. To coordinate mode-dependent objectives, we pose an event-driven switching NMPC problem whose optimization structure adapts to the active phase, following the general multi-mode NMPC structure in \cite{rawlings2020,mcallister2022}:
\begin{equation}
	\label{eq:switchMPC}
	\begin{aligned}
		\min_{u \in S(\Delta)} &J(x(t_k),U) = \int_{t_k}^{t_{k}+N\Delta} l_{\mu(x(\tau))}(x(\tau), u(\tau)) d\tau + F_{\mu(x(t_{k}+N\Delta))}(x(t_{k}+N\Delta)) \\
		\text{s.t.} \quad & \dot{x}(t) = f(x(t), u(t)), \quad t \in [t_k, t_{k}+N\Delta] \\
		& g_{\mu(x(t))}(x(t), u(t)) \leq 0, \quad t \in [t_k, t_{k}+N\Delta] \\
		& x(t_k) = x_k \text{ (initial condition)} \\
		& u(t) \in \mathbb{U}, \quad t \in [t_k, t_{k}+N\Delta] \, .
	\end{aligned}
\end{equation}
where the continuous-time dynamics $\dot{x}(t) = f(x(t), u(t))$ describe the evolution of the state vector $x(t) \in \mathbb{R}^n$ collecting concentrations, temperature, and pressure, starting from the initial measurement $x_k$, driven by manipulated inputs $u(t) \in \mathbb{R}^m$. The decision variable $U=\{u_j\}_{j=0}^{N-1} \in S(\Delta)$ represents the piecewise-constant control sequence sampled with period $\Delta$, where $S(\Delta)$ denotes the set of admissible trajectories subject to input bounds $\mathbb{U}$. The functions $l_{\mu}(\cdot)$, $F_{\mu}(\cdot)$, and $g_{\mu}(\cdot)$ define the specific cost and constraints active for the current mode $\mu \in P=\{1,\ldots,M\}$. The mode selector $\mu: \mathbb{R}^n \rightarrow P$ maps states to operational modes. Switching times are not optimization variables; instead, the controller shapes the state evolution so that $\mu(x(t))$ triggers the appropriate mode specified in Table~\ref{tab:poly_modes}. This state-driven mechanism tolerates variations in process parameters but introduces abrupt changes in cost and constraint landscapes across mode boundaries. This nonsmooth optimization structure is addressed by the framework developed in the subsequent sections.
\section{Real-time NMPC Framework for Switched Systems}
\label{sec:rt}
\subsection{Smoothing Approximation for Switching Functions}
Equation~\eqref{eq:switchMPC} changes structure whenever the active mode switches, breaking the smoothness assumptions required by gradient-based NLP solvers. 
The changes are discontinuous by binary indicators, which are a function of the switching surface $h_i(x)$.
Each switching surface $h_i(x)$ therefore introduces a binary indicator
\begin{equation}
	\label{eq:heaviside}
p_i(x) = H(h_{i}(x)) = \begin{cases}
	1 & \text{if } h_{i}(x) \geq 0 \\
	0 & \text{if } h_{i}(x) < 0
\end{cases}, \quad i=1,\ldots,M,
\end{equation}
which we replace by a differentiable surrogate of the form
\begin{equation}
	\label{eq:smooth}
	p_i^{\text{smooth}}(x) := \sigma\bigl(\alpha (h_i(x)-\beta)\bigr)
\end{equation}
where $\sigma(z):=\bigl(1+\exp({-z})\bigr)^{-1}$ denotes the logistic function. $\alpha$ and $\beta$ are tunable parameters representing the slope and bias of the switching surface, respectively.

The logistic surrogate~\eqref{eq:smooth} converts the discontinuous step to a continuously differentiable transition over a band of width $O(\alpha^{-1})$ centered on the switching surface. the logistic surrogate replaces the discontinuous jump with a smooth "S-shape" curve. The transition (interval where \eqref{eq:smooth} changes from 0 to 1) is controlled by the slope parameter $\alpha$ and the bias parameter $\beta$. 

We employ separate relaxations for the cost and the constraints: the cost-side indicator $p_i^{\text{obj}}$ uses no offset so that the economic stage cost begins to contribute as the trajectory approaches the switching surface, whereas the constraint-side indicator $p_i^{\text{cons}}$ is shifted by $\beta_i$ so that constraints are enforced conservatively until the surface is actually crossed:
\begin{equation}
    p_i^{\text{obj}}(x) = \sigma\bigl(\alpha_i h_i(x)\bigr) 
\end{equation}

\begin{equation}
    p_i^{\text{cons}}(x) = \sigma\bigl(\alpha_i (h_i(x)-\beta_i)\bigr) 
\end{equation}

This leads to our smoothed NMPC formulation:
\begin{equation}
	\label{eq:smoothNMPC}
	\begin{aligned}
		\min_{U\in S(\Delta)}\ 
		&J^{\mathrm{smooth}}(x(t_k),U)
		 := \int_{t_k}^{t_k+N\Delta}\!
			\sum_{i=1}^{M} p_i^{\mathrm{obj}}\!\bigl(x(\tau)\bigr)\,
			l_i\!\bigl(x(\tau),u(\tau)\bigr)\, \mathrm{d}\tau \\
		&\qquad\quad
		 + \sum_{i=1}^{M} p_i^{\mathrm{obj}}\!\bigl(x(t_k+N\Delta)\bigr)\,
			F_i\!\bigl(x(t_k+N\Delta)\bigr),\\
    \text{s.t.}\quad
    &\dot{x}(t)=f\bigl(x(t),u(t)\bigr),\\
    &p_i^{\mathrm{cons}}\bigl(x(t)\bigr)\,g_i\bigl(x(t),u(t)\bigr)\le0,\quad i=1,\ldots,M\\
    &x(t_k)=x_k,\quad u(t)\in\mathbb{U},\\
    &u(t)=u_j\quad\forall t\in[t_k+j\Delta,t_k+(j+1)\Delta), j = 0,1,...,N-1.
	\end{aligned}
\end{equation}

Parameter selection for $\alpha_i$ and $\beta_i$, together with approximation error bounds, is discussed in Section~\ref{sec:theoretical_analysis}.
\subsection{Real-Time NMPC Framework Integration}
\label{sec:framework}

To achieve real-time feasibility for the stiff polymerization process, the proposed framework integrates the smoothing approximation with three established numerical strategies. For the sake of brevity, we formulate the NMPC problem \eqref{eq:smoothNMPC} into a standard NLP:
\begin{equation}
    \label{eq:nlp_formulation}
    \begin{aligned}
        \min_{w } \quad F(w, x_k) \quad
        \text{s.t.} \quad &c(w, x_k) = 0 \quad
        g(w, x_k) \leq 0
    \end{aligned}
\end{equation}
where $w = [U^T]^T$ and $x_k$ is the initial state. $F(.)$ is the objective function, $c(.)$ are the equality constraints , and $g(.)$ are the inequality constraints.

Following the recommended numerical strategies in literature, we adopt:

\begin{itemize}
	\item \textbf{Variable Scaling:} To address ill-conditioning (e.g., $w_1 \sim 10^3$ vs. $w_2 \sim 10^{-7}$), we employ a diagonal scaling matrix $D = \text{diag}(d_1, \dots, d_n)$. The NLP solver operates on scaled variables $\tilde{w} = D^{-1}w$ \cite{biegler2009}. This external scaling complements the solver's internal scaling (e.g., equilibration of Hessian rows in IPOPT) by normalizing the decision space before derivatives are evaluated.
        \begin{equation}
        \min_{\tilde{w}} F(D\tilde{w}, x_k) \quad \text{s.t.} \quad c(D\tilde{w}, x_k) = 0, \quad g(D\tilde{w}, x_k) \leq 0.
        \label{eq:nlp_scaled}
    \end{equation}

    \item \textbf{Advanced-step NMPC (asNMPC):} Optimization is split into a background phase and an online sensitivity update \cite{zavala2009a}. The background problem Eq.~\eqref{eq:smoothNMPC} is solved for the predicted state. Upon measuring $x_k$, the optimal primal-dual vector $s^*$ (which contains the decision variables $w^*$ and the Lagrange multipliers) is corrected via:
    \begin{equation}
        \tilde{s}(x_{k})=s^{*}(x_{k-1})+\mathcal{M}^{-1}\mathcal{N}(x_{k}-x_{k-1}),
        \label{eq:sens_update}
    \end{equation}
    where $\mathcal{M}$ is the KKT matrix (Lagrangian Hessian and constraint Jacobians) and $\mathcal{N}$ represents the sensitivity of the optimality conditions with respect to parameter $x_k$. A line search on $\tau \in [0,1]$ projects this prediction back to the feasible set $\mathbb{U}$:
    \begin{equation}
        \hat{s}(x_{k})=s^{*}(x_{k-1})+\tau(\tilde{s}(x_{k})-s^{*}(x_{k-1})).
        \label{eq:proj_update}
    \end{equation}

    \begin{figure}[ht]
    \centering
    \begin{subfigure}[b]{0.88\columnwidth}
        \centering
        \includegraphics[width=\columnwidth]{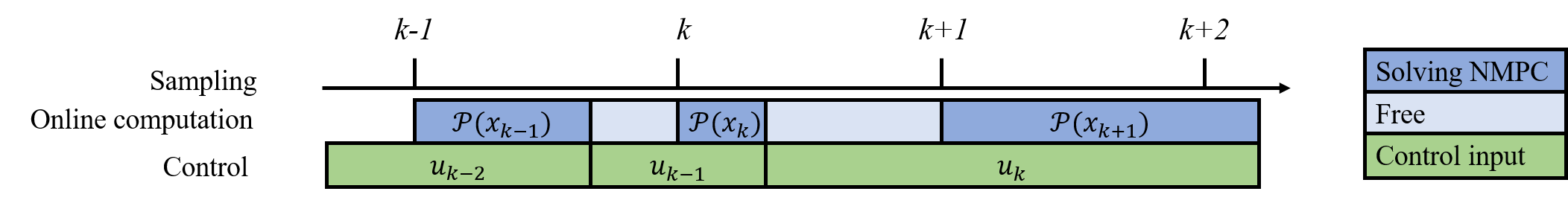}
        \caption{Standard NMPC}
        \label{fig:normNMPC}
    \end{subfigure}
    \begin{subfigure}[b]{0.88\columnwidth}
        \centering
        \includegraphics[width=\columnwidth]{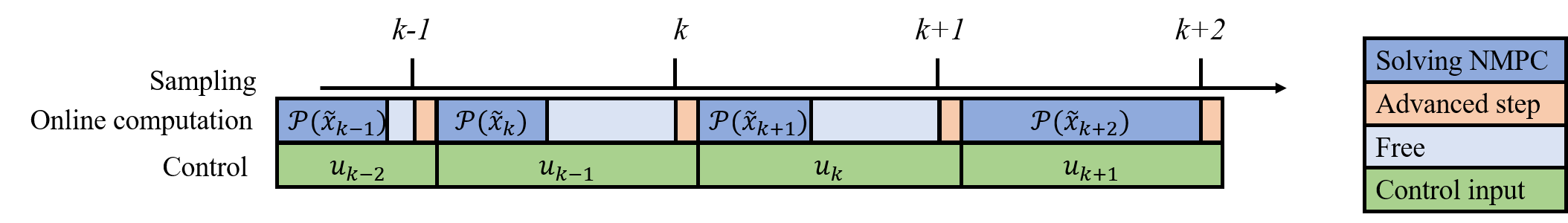}
        \caption{Advanced-step NMPC}
        \label{fig:asNMPC}
    \end{subfigure}
    \caption{Timing diagrams illustrating the computational delay handling. Standard NMPC (a) applies control $u_k$ only after optimization concludes, introducing a delay $\delta$. Advanced-step NMPC (b) shifts the optimization to the background, allowing instant correction via sensitivity updates when $x_k$ becomes available.}
    \label{fig:nmpc_timing}
    \end{figure}

    \item \textbf{Fixed Iteration Budget:} Following \cite{wang2010}, we impose a strict cap $K_{\max}$ on interior-point iterations. The solver returns the current iterate when the limit is reached. Although strict convergence is not guaranteed, the sensitivity update acts as a restorative step since the line search \eqref{eq:proj_update} ensures a feasible control is applied. Because the warm-start from the previous step is provided, which is close to the optimal solution, the iterates typically converge quickly.
\end{itemize}

Algorithm \ref{alg:rt_nmpc} summarizes the complete workflow, alternating between background prediction and online correction. The individual contributions of smoothing, advanced-step execution, scaling, and iteration capping are quantified in the ablation study of Section~\ref{sec:case}.

\begin{algorithm}[t]
\caption{Real-time smoothed NMPC with scaling and advanced-step warm start}
\label{alg:rt_nmpc}
\begin{algorithmic}[1]
\State \textbf{Input:} smoothing parameters $\{\alpha_i,\beta_i\}$, scaling matrix $D$, prediction horizon $N$, sampling time $\Delta$, iteration cap $K_{\max}$
\State \textbf{Initialisation:} solve the scaled smoothed NLP $\mathcal{P}(D^{-1}x_0)$ until convergence; store primal-dual solution $s^\star(q_0)$ with $q_0=D^{-1}x_0$ and the KKT factorisation.
\For{$k = 0,1,\ldots$}
    \State Measure/estimate $x_k$; set $q_k = D^{-1}x_k$.
    \State \textbf{Online correction}
        \Statex \quad (a) Apply sensitivity update $\tilde{s}(q_k)$ via \eqref{eq:sens_update}.
        \Statex \quad (b) Project to feasible set by selecting the largest $\tau\in[0,1]$ that keeps \eqref{eq:proj_update} within $\mathbb{U}$.
        \Statex \quad (c) Extract first control move $u_k$ from $\hat{s}(q_k)$ and apply to the plant.
    \State \textbf{Background prediction}
        \Statex \quad (a) Integrate $\dot x = f(x,u_k)$ from $x_k$ over one sampling period $\Delta$ to obtain the predicted state $\tilde{x}_{k+1}$ and set $q_{k+1}=D^{-1}\tilde{x}_{k+1}$.
        \Statex \quad (b) Warm-start NLP $\mathcal{P}(q_{k+1})$ with $\hat{s}(q_k)$ and iterate IPOPT for at most $K_{\max}$ steps; if convergence not reached, keep iterate as new $s^\star(q_{k+1})$.
        \Statex \quad (c) Update sensitivity matrices $\mathcal{M}^{-1}\mathcal{N}$ using the active constraint set of $\mathcal{P}(q_{k+1})$.
\EndFor
\end{algorithmic}
\end{algorithm}

\section{Theoretical Analysis of the Smoothing Strategy}
\label{sec:theoretical_analysis}


We begin by recalling the discontinuous mode indicator $p_i(x)=H(h_i(x))$ from~\eqref{eq:heaviside} and the logistic surrogates $p_i^{\text{obj}}(x)$, $p_i^{\text{cons}}(x)$ from~\eqref{eq:smooth}, governed by the smoothing gains $\alpha_i>0$, bias $\beta_i\ge 0$, and target accuracy $\epsilon\in(0,\tfrac{1}{2})$. The resulting transition region, where the surrogates disagree with the true indicator, is illustrated in Figure~\ref{fig:smoothing_parameter_schematic}; additional elements of the figure are discussed in the following subsections.

Building on these objects, we establish the following results: (i) a characterisation of the transition region (Lemma~\ref{lem:transition_region}), (ii) a bound on the dwell time inside it (Lemma~\ref{lem:dwell_time}), (iii) constraint inclusion guarantees (Lemma~\ref{lem:feasible_inclusion}), (iv) a cost approximation error bound (Lemma~\ref{lem:cost_function_bound}), and (v) an overall bound on the maximum approximation error of the smoothed NMPC (Theorem~\ref{thm:approximation_error}). Only the statements are given here; the full proofs are collected in Appendix~\ref{app:proofs}. Practitioners primarily interested in implementation may skip directly to Subsection~\ref{subsec:tuning_guidelines}, where the theoretical results justify the practical parameter-selection guidance.

\begin{figure}[H]
\centering
\includegraphics[width=0.98\columnwidth]{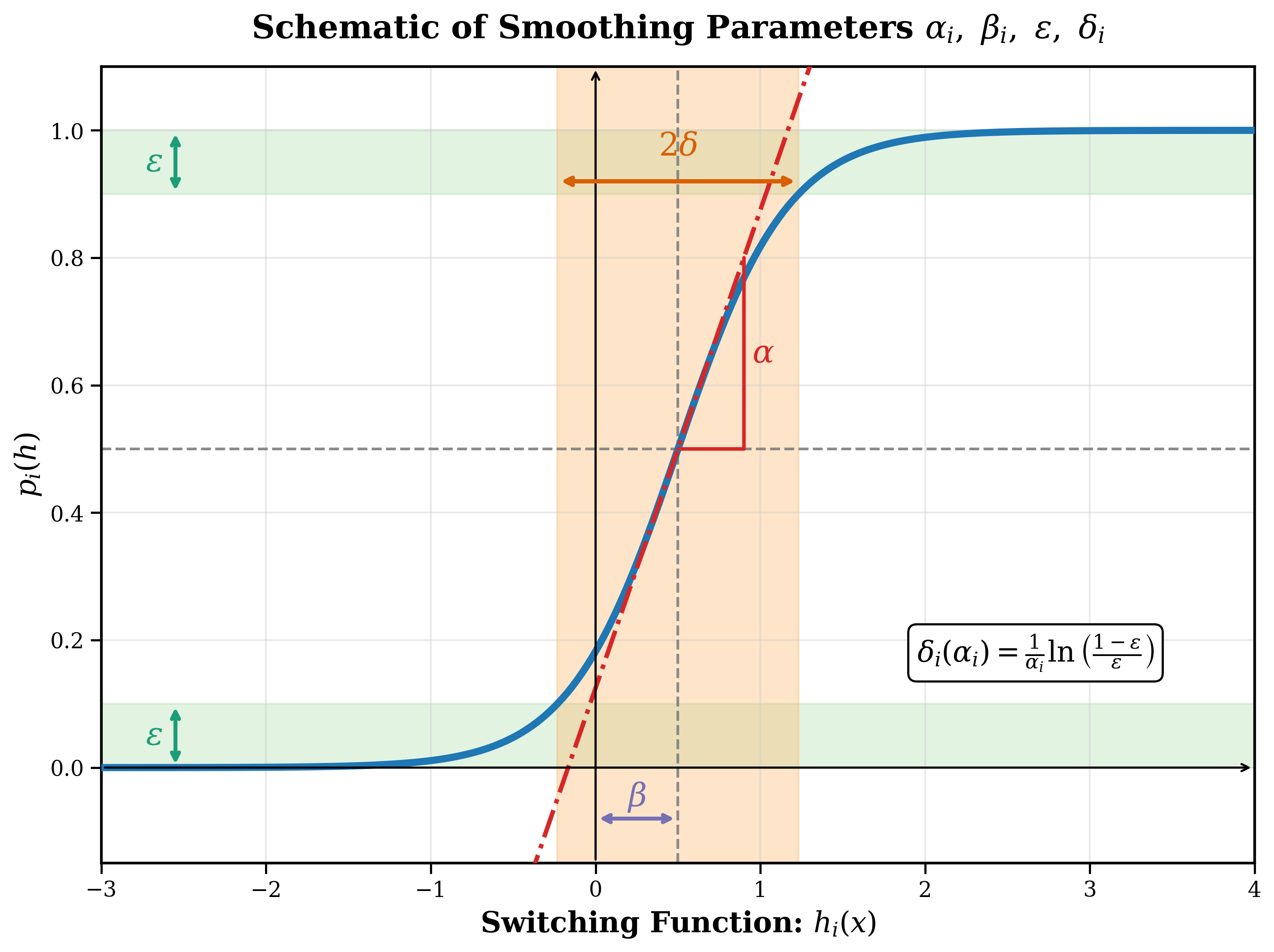}
\caption{Geometric interpretation of the smoothing approximation. The logistic function $\sigma(\alpha(h(x)-\beta))$ creates a smooth transition band of width $\delta_i(\alpha_i)=\frac{1}{\alpha_i}\ln\left(\frac{1-\varepsilon}{\varepsilon}\right)$ around the switching surface $h(x)=\beta$, replacing the discontinuous step. The transition region is indicated by the shaded yellow band.}
\label{fig:smoothing_parameter_schematic}
\end{figure}

\begin{lemma}[Transition Region Characterization]
\label{lem:transition_region}
Let $\epsilon\in(0,\tfrac{1}{2})$. The transition region for $p_i^{\text{obj}}$ is the set
\[
	\mathcal{R}_i(\alpha_i)
	:= \{\, x : |p_i^{\text{obj}}(x)-p_i(x)| > \epsilon \,\}
	= \{\, x : -\delta_i(\alpha_i) < h_i(x) < \delta_i(\alpha_i) \,\},
\]
where the half-width is
\[
	\delta_i(\alpha_i) = \frac{1}{\alpha_i}\,\ln\!\left(\frac{1-\epsilon}{\epsilon}\right).
\]
Hence the total width satisfies $2\delta_i(\alpha_i)=O(\alpha_i^{-1})$.
\end{lemma}
\noindent\textit{Proof.} See Appendix~\ref{app:proofs}.

\begin{assumption}[System dynamics]
\label{ass:system_dynamics}
The vector field $f(x,u)$ is locally Lipschitz and has linear growth on compact sets.
The admissible input set $\mathbb{U}$ is compact and convex. Inputs are implemented as piecewise-constant signals over intervals of length $\Delta>0$, matching the NMPC discretisation. Hence trajectories are unique and
forward complete under piecewise-constant inputs.
\end{assumption}

\begin{assumption}[Switching surface regularity and uniform transversality]
\label{ass:switching_surface}
For each switching surface $h_i(x)=0$ the gradient is bounded, and there exist constants
$\bar{\delta}_i>0$, ${\nu}_i>0$, and an integer $K_i\ge1$ such that closed-loop trajectories generated
by admissible inputs $u(t)\in\mathbb{U}$ satisfy
\[
  |h_i(x)|\le \bar{\delta}_i
  \;\Rightarrow\;
  \big|\nabla h_i(x)^\top f(x,u)\big|
  \ge {\nu}_i,
\]
hence trajectories cross each surface transversally within $|h_i(x)|\le\bar{\delta}_i$, and at most
$K_i$ crossings occur on any finite horizon.
\end{assumption}
\begin{remark}[Practical design of switching surfaces]
\label{rem:switching_design}
This theoretical analysis informs practical design: switching surfaces defined by variables that evolve monotonically (e.g., conversion, cumulative feed, time) naturally satisfy the transversality condition (Assumption~\ref{ass:switching_surface}) because $\nabla h_i^\top f$ does not vanish. In contrast, surfaces defined by regulated variables (e.g., temperature around a setpoint) may violate transversality ($\nabla h_i^\top f \approx 0$), causing the denominator in Lemma~\ref{lem:dwell_time} to approach zero and the dwell time bound to diverge. Therefore, we recommend defining phase transitions using integral or monotonic states to guarantee robust numerical behavior.
\end{remark}

\begin{lemma}[Transition Region Dwell Time]
\label{lem:dwell_time}
Under Assumptions \ref{ass:system_dynamics} and \ref{ass:switching_surface}, for each surface $i$, and $\bar{\delta}_i\ge \delta_i(\alpha_i) >0$, the time the trajectory spends in $\mathcal{R}_i(\alpha_i)=\{x: |h_i(x)|\le \delta_i(\alpha_i)\}$
satisfies
\begin{equation}
  \tau_i(x(t),\delta_i) \ \le\ K_i\,\frac{2\,\delta_i(\alpha_i)}{{\nu}_i},
\end{equation}
where $\tau_i(x(t),\delta_i)$ is the Lebesgue measure of the visit set on $[0,T]$.
\end{lemma}
\noindent\textit{Proof.} See Appendix~\ref{app:proofs}.

\subsection{Smoothing Approximation Analysis}
We now collect the numerical tolerances and local regularity conditions invoked by the smoothing analysis.
\begin{assumption}[Solver resolution]
	\label{ass:solver_tol}
	The NLP solver uses finite tolerances $\tau_{\mathrm{feas}}>0$ (feasibility) and
	$\tau_{\mathrm{opt}}>0$ (optimality). Constraints are considered satisfied whenever
	$\max\{0,\,g_i(x,u)\}\le \tau_{\mathrm{feas}}$.
\end{assumption}

\begin{assumption}[Smoothing accuracy]
\label{ass:error}
For each surface there exists $\epsilon_i\in(0,\tfrac{1}{2})$ and a neighbourhood
$\mathcal{N}_i$ of $\{x:h_i(x)=0\}$ such that
$|p_i^{\text{obj}}(x)-p_i(x)|\le\epsilon_i$ and
$|p_i^{\text{cons}}(x)-p_i(x)|\le\epsilon_i$ whenever $x\notin\mathcal{N}_i$, with
$\epsilon_i \le \tau_{\mathrm{feas}}$.
\end{assumption}
The geometry in Figure~\ref{fig:smoothing_parameter_schematic} corresponds to this assumption, where $\mathcal{N}_i$ aligns with the transition bandwidth $2\delta$ and $\epsilon_i$ with the tolerance bands.
\begin{assumption}[Constraint regularity near switching surfaces]
\label{ass:constraint_regularity}
There exist constants $L_{g,i}>0$, $\kappa_i>0$, and $G_{g,i}>0$ such that for any
$|h_i(x)|\le \delta$ (for sufficiently small $\delta>0$) there exists $x_i^{0}$ on $h_i(x)=0$ with
$\|x - x_i^{0}\| \le \kappa_i |h_i(x)|$ and
\[
  g_i(x,u) \ \le\ g_i(x_i^{0},u) + L_{g,i}\,|h_i(x)|,
  \qquad |g_i(x,u)| \ \le\ G_{g,i}.
\]
These bounds hold in the neighbourhoods $\mathcal{N}_i$ from Assumption~\ref{ass:error}, and the
solver tolerance satisfies $\epsilon_i G_{g,i} \le \tau_{\mathrm{feas}}$.
\end{assumption}
Because $\delta_i(\alpha_i)=O(\alpha_i^{-1})$, the design requirement
$L_{g,i}\bigl(\beta_i+\delta_i(\alpha_i)\bigr)\le\tau_{\mathrm{feas}}$ is enforceable by increasing
$\alpha_i$, shrinking the neighbourhood where constraint blending occurs.

\begin{lemma}[Feasibility Inclusions for Smoothed Constraints]
\label{lem:feasible_inclusion}
Fix $\epsilon\in(0,\tfrac12)$ and define
$\delta_i(\alpha_i)=\alpha_i^{-1}\ln\!\bigl(\tfrac{1-\epsilon}{\epsilon}\bigr)$.
For any $\beta_i\in[0,\delta_i(\alpha_i)]$
Define the feasible sets
\[
	\mathcal S_{\text{orig}}
	:=\{\,x\ \mid\ \forall i,\ p_i(x)\,g_i(x,u)\le 0\,\},
\]
\[
	\mathcal S_{\text{smooth}}^{\tau}
	:=\{\,x\ \mid\ \forall i,\ p_i^{\text{cons}}(x)\,g_i(x,u)\le \tau_{\mathrm{feas}}\,\}.
\]
Assume the solver feasibility tolerance $\tau_{\mathrm{feas}}>0$
(Assumption~\ref{ass:solver_tol}) and the local bounds in
Assumption~\ref{ass:constraint_regularity}, together with
$\epsilon\,G_{g,i}\le\tau_{\mathrm{feas}}$ and
$L_{g,i}\bigl(\beta_i+\delta_i(\alpha_i)\bigr)\le \tau_{\mathrm{feas}}$.
Then
\[
	\mathcal S_{\text{orig}}\ \subseteq\ \mathcal S_{\text{smooth}}^{\tau}
	\ \subseteq\ \mathcal S_{\text{orig}}^{+}(\tau),
\]
where
\[
	\mathcal S_{\text{orig}}^{+}(\tau)
	:=\Bigl\{\,x\ \Big|\ \forall i,\
	p_i(x)\,g_i(x,u)\le {\tau_{\mathrm{feas}}}/{\bar{p}_i(\alpha_i,\beta_i)}\,\Bigr\},
\]
\[
	\bar{p}_i(\alpha_i,\beta_i):=\inf_{h\ge 0} p_i^{\text{cons}}(h)
	=\sigma\!\bigl(-\alpha_i\beta_i\bigr).
\]
In particular, \emph{up to the solver tolerance} the smoothed and original feasibility
tests are numerically indistinguishable.
\end{lemma}
\noindent\textit{Proof.} See Appendix~\ref{app:proofs}.

\begin{remark}[Feasibility consistency]
\label{rem:feas_consistency}
This lemma establishes that if $\beta_i$ and $\alpha_i$ are tuned according to the guidelines, the smoothed feasible set $\mathcal{S}_{\text{smooth}}^{\tau}$ is effectively indistinguishable from the original set $\mathcal{S}_{\text{orig}}$. Specifically, any point feasible for the original problem is feasible for the smoothed problem (up to tolerance), and any point feasible for the smoothed problem lies within the $\tau$-tolerance band of the original constraints. Thus, the solver sees no "phantom" infeasibility nor significant constraint violation.
\end{remark}
	
\begin{remark}[Risks of excessive constraint shifting]
\label{rem:beta_greater_delta}
Although increasing $\beta_i$ enlarges the feasible set $\mathcal S_{\text{smooth}}^{\tau}$ (as implied by Lemma~\ref{lem:feasible_inclusion}), setting $\beta_i > \delta_i(\alpha_i)$ is hazardous because it breaks the tightness of the inclusion $\mathcal S_{\text{smooth}}^{\tau} \subseteq \mathcal S_{\text{orig}}^{+}(\tau)$. Specifically, it creates a ``blind spot'' where the objective mode is active ($p_i^{\text{obj}} \approx 1$) but the constraint strength $\bar{p}_i$ is negligible. This manifests as: (i) \textbf{Tolerance Inflation}: The bound in Lemma~\ref{lem:feasible_inclusion} becomes loose ($1/\bar{p}_i \gg 1/\epsilon$), allowing physically unsafe violations; and (ii) \textbf{Solver Blindness}: Gradients vanish in the gap, preventing the optimizer from detecting and reacting to the approaching constraint. Thus, $\beta_i = \delta_i(\alpha_i)$ is the optimal tradeoff: it maximizes the feasible set size without compromising the strict enforcement guaranteed by Lemma~\ref{lem:feasible_inclusion}.
\end{remark}

\begin{lemma}[Cost function approximation error bound]
\label{lem:cost_function_bound}
Consider a system satisfying Assumption~\ref{ass:system_dynamics}. Given a control trajectory $U$, if
the resulting state trajectory crosses surfaces that satisfy Assumption~\ref{ass:switching_surface},
then the cost difference between the original NMPC problem \eqref{eq:switchMPC} and its smoothed
counterpart \eqref{eq:smoothNMPC} using $p_i^{\text{obj}}(x)$ obeys
\begin{equation}
	|J^{\text{smooth}}(x(t_k),U) - J(x(t_k),U)| \leq \gamma(\alpha_{\min}^{-1}),
\end{equation}
where $\alpha_{\min} := \min_{i=1,\ldots,M} \alpha_i$,
$C_i := \sup\{\|l_i(x,u)\| : |h_i(x)|\le \delta_i(\alpha_i),\, u\in\mathbb{U}\}<\infty$,
$c_i := \frac{2K_i C_i}{{\nu}_i}\ln\frac{1-\epsilon}{\epsilon}$, and the class-$\mathcal{K}$ function
$\gamma(r)$ is defined explicitly as $\gamma(r) := \bigl(\sum_{i=1}^M c_i\bigr) r$.
\end{lemma}
\noindent\textit{Proof.} See Appendix~\ref{app:proofs}.

The feasible-set inclusions and cost bound combine into the following consequence for practical solvers.

\begin{theorem}[Approximation Error of Smoothed NMPC]
\label{thm:approximation_error}
Suppose (i) Assumptions \ref{ass:system_dynamics} and \ref{ass:switching_surface} hold;
(ii) the original NMPC \eqref{eq:switchMPC} satisfies LICQ and SOSC at its optimizer $U^\ast(x)$
and thus enjoys local second-order growth: there exist $\mu>0$ and $r>0$ such that for all feasible
$U$ with $\|U-U^\ast\|\le r$,
\[
  J(x,U)-J(x,U^\ast)\ \ge\ \tfrac{\mu}{2}\,\|U-U^\ast\|^2;
\]
(iii) offsets are chosen as $\beta_i=\delta_i(\alpha_i)$ so that
Lemma~\ref{lem:feasible_inclusion} yields
$\mathcal{S}_{\text{orig}} \subseteq \mathcal{S}_{\text{smooth}}^{\tau}(\alpha,\beta)
\subseteq \mathcal{S}_{\text{orig}}^{+}(\tau,\alpha,\beta)$ with
$\bar{p}_i(\alpha_i,\beta_i)=1-\epsilon_i$.
Then, given the same initial state $x$,
\begin{equation}
  \|U^{\text{smooth}*}(x) - U^*(x)\|
  \ \le\ \sqrt{\frac{4\,\gamma(\alpha_{\min}^{-1})}{\mu}}.
\end{equation}
\end{theorem}
\noindent\textit{Proof.} See Appendix~\ref{app:proofs}.

\begin{remark}[$\epsilon$-Suboptimality]
\label{rem:epsilon_suboptimality}
Theorem~\ref{thm:approximation_error} and Lemma~\ref{lem:cost_function_bound} formally characterize the proposed method as an $\epsilon$-suboptimal approximation scheme. The bound $2\gamma(\alpha_{\min}^{-1})$ represents the duality gap introduced by smoothing: as $\alpha \to \infty$, the gap vanishes. For finite $\alpha$ (e.g., $\alpha \approx 50$), the solution is ``$\epsilon$-optimal'' with respect to the original nonsmooth problem, providing a rigorous justification for using the smoothed control in practice.
\end{remark}

\begin{figure}[ht]
\centering
\includegraphics[width=\columnwidth]{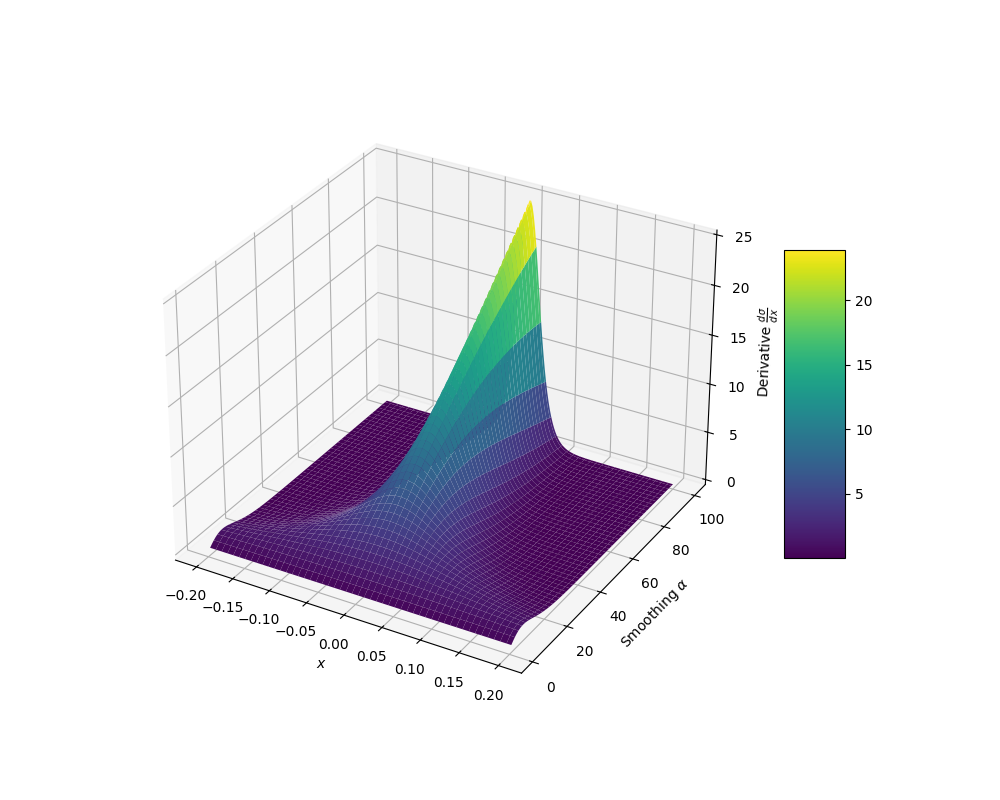}
\caption{Gradient landscape of the smoothing function derivative $\frac{d\sigma}{dx} = \alpha \sigma(1-\sigma)$. As $\alpha$ increases, the derivative transforms into a numerical Dirac spike, illustrating the trade-off between approximation accuracy and numerical tractability.}
\label{fig:gradient_spike}
\end{figure}

\begin{remark}[Numerical stability and practical tuning]
The numerical failure at high $\alpha$ can be explained by analyzing the gradient landscape of the smoothed constraints. The derivative of the switching function is given by $\nabla_x p(x) \propto \alpha \sigma(\cdot)(1-\sigma(\cdot))$. As illustrated in Figure~\ref{fig:gradient_spike}, increasing $\alpha$ transforms the gradient into a narrow, high-amplitude impulse (approximating a Dirac delta). This creates two numerical issues for the NLP solver: (i) \textit{Gradient Vanishing (Step-over)}: When the transition region width ($\sim 1/\alpha$) becomes smaller than the solver's current line-search step size, the algorithm may ``step over'' the switching surface entirely, seeing zero gradients on both sides. (ii) \textit{Gradient Explosion \& Ill-conditioning}: If an iterate lands exactly within the narrow transition band, the Hessian terms scale with $\alpha^2$. For very large $\alpha$ (e.g., $\alpha > 1000$), this introduces extreme curvature, causing the Condition Number of the KKT matrix to spike. This ill-conditioning triggers the restoration failures observed in Section~\ref{subsec:toy_verification}.
\end{remark}

\subsection{Practical Tuning Guidelines}
\label{subsec:tuning_guidelines}
The theoretical analysis yields three practical guidelines for selecting smoothing parameters:

\begin{enumerate}
  \item \textbf{Selection of Smoothing Gain $\alpha_i$ (The Accuracy-Stability Trade-off):}
  Theoretically, selecting $\alpha_i$ is a trade-off governed by the error bound $\gamma(\alpha_{\min}^{-1})$.
  \textit{Lower Bound (Accuracy):} $\alpha_i$ must be sufficiently large to minimize the optimality gap. Our analysis shows that $\alpha_i < 10$ often leads to excessive conservatism (premature constraint activation).
  \textit{Upper Bound (Stability):} Critically, $\alpha_i$ must be bounded to prevent numerical breakdown. As illustrated in Figure~\ref{fig:gradient_spike}, excessively large gains ($\alpha_i > 100$) create needle-like gradients that cause gradient-based solvers to fail due to search step-over (vanishing gradients) or Hessian ill-conditioning (exploding gradients).
  \textit{Recommendation:} We recommend a range of $\alpha_i \in [10, 100]$. In our experiments, $\alpha_i \approx 50$ consistently provided negligible approximation error ($<0.1\%$) while maintaining robust solver convergence.

  \item \textbf{Alignment of Constraint Shift $\beta_i$ (Feasibility Guarantee):}
  To ensure the smoothed problem does not violate the original hard constraints, the shift parameter should be set as $\beta_i = \delta_i(\alpha_i) = \frac{1}{\alpha_i}\ln(\frac{1-\epsilon}{\epsilon})$. This setting aligns the $p_i^{\text{cons}} \approx \epsilon$ contour with the original switching surface $h_i(x)=0$. Setting $\beta_i < \delta_i(\alpha_i)$ introduces a safety margin $\eta_i > 0$ for additional conservatism under uncertainty. Setting $\beta_i > \delta_i(\alpha_i)$ is not recommended as it may permit slight violations of the original constraints. (See Remark~\ref{rem:beta_greater_delta})

  \item \textbf{Solver Tolerance Matching:}
  The transition width parameter $\epsilon$ should be chosen compatible with the NLP solver's feasibility tolerance $\tau_{\text{feas}}$. Selecting $\epsilon \approx \tau_{\text{feas}}$ (e.g., $10^{-4}$ to $10^{-6}$) ensures that the ``blurred'' region of the constraint is numerically indistinguishable from the strict boundary by the solver logic.
\end{enumerate}

\subsection{Illustrative Example: Numerical Verification}
\label{subsec:toy_verification}
To strictly validate the theoretical bounds derived in Section 4.1 and demonstrate the practical impact of the smoothing parameters, we verify the framework on a constrained double integrator system. It is chosen because this linear benchmark allows for the computation of a global optimum, serving as a rigorous ground truth for error analysis.


The system models a mass ($m=1$) transitioning between a high-speed zone and a restricted-speed zone. The dynamics and optimal control problem are formulated as:
\begin{subequations}
\begin{align}
    \min_{u(t)} \quad & J = \int_{0}^{T} \left( (v(t) - 10)^2 + 0.1 u(t)^2 \right) dt \\
    \text{s.t.} \quad & \dot{p}(t) = v(t), \quad \dot{v}(t) = u(t), \\
    & p(0) = 0, \quad v(0) = 0, \\
    & u(t) \in [-1, 1], \\
    & v(t) \leq \begin{cases} 10.0 & \text{if } p(t) < 5 \quad (\text{High-Speed Zone}) \\ 1.0 & \text{if } p(t) \ge 5 \quad (\text{Restricted Zone}) \end{cases} \label{eq:velocity_constraint}
\end{align}
\end{subequations}
In the smoothed NMPC formulation, the state-dependent velocity constraint (\ref{eq:velocity_constraint}) is replaced by the sigmoid-blended approximation:
\begin{equation}
    v(t) \leq 10.0 (1 - \sigma(\alpha(p(t) - 5 - \beta))) + 1.0 \sigma(\alpha(p(t) - 5 - \beta))
\end{equation}

Since the system dynamics are linear and the cost is quadratic, the original non-smooth problem can be formulated exactly as a \textbf{Mixed-Integer Quadratic Programming (MIQP)} problem using the Big-M method to model the logical switching. We solved this MIQP to global optimality using \textbf{Gurobi} with a fine discretization ($N=80$) to obtain the exact benchmark trajectories $u^*$.


The verification results are summarized in Figure~\ref{fig:verification_suite}, confirming the four key theoretical properties derived in Section 4.1:

\begin{itemize}
    \item \textbf{Panel A (Transition Dwell Time):} The time spent in the transition region ($\epsilon$-band) is plotted against the smoothing gain $\alpha$. The data perfectly follows a $\tau \propto 1/\alpha$ trend (log-log slope $\approx -1$), empirically validating \textbf{Lemma 2}. This confirms that increasing $\alpha$ linearly reduces the ``blurring'' window of the switch.
    \item \textbf{Panel B (Feasibility \& Shift Tuning):} Comparing the smoothed solutions against the hard limit $v \le 1.0$ (for $p \ge 5$) reveals the necessity of the parameter $\beta$. The untuned case ($\beta=0$, red line) violates the constraint. In contrast, using our derived tuning rule $\beta_i = \delta_i(\alpha_i)$ (green line) results in reduced violation which rapidly drops to a negligible level ($\sim 10^{-7}$) as $\alpha$ increases, validating the inclusion logic in \textbf{Lemma 3}.
    \item \textbf{Panel C (Cost Error):} The \textit{absolute} cost difference $|J^{\mathrm{smooth}} - J^*|$ decays as $O(1/\alpha)$, confirming \textbf{Lemma 4}. Note that the theoretical upper bound (dashed line) correctly captures the convergence rate (-1 slope) but is numerically conservative (an order of magnitude larger) due to worst-case Lipschitz constant estimates.
    \item \textbf{Panel D (Input Convergence):} The $L_2$-norm difference $\|U^{\mathrm{smooth}} - U^*\|$ decreases monotonically, consistent with \textbf{Theorem 1}. However, for $\alpha > 100$, the solver fails to converge within the iteration limit due to the ill-conditioning described in Remark 4. This confirms that while theoretical error vanishes as $\alpha \to \infty$, practical limits impose a finite ''sweet spot'' for tuning.
\end{itemize}

\begin{figure}[H]
\centering
\includegraphics[width=\textwidth]{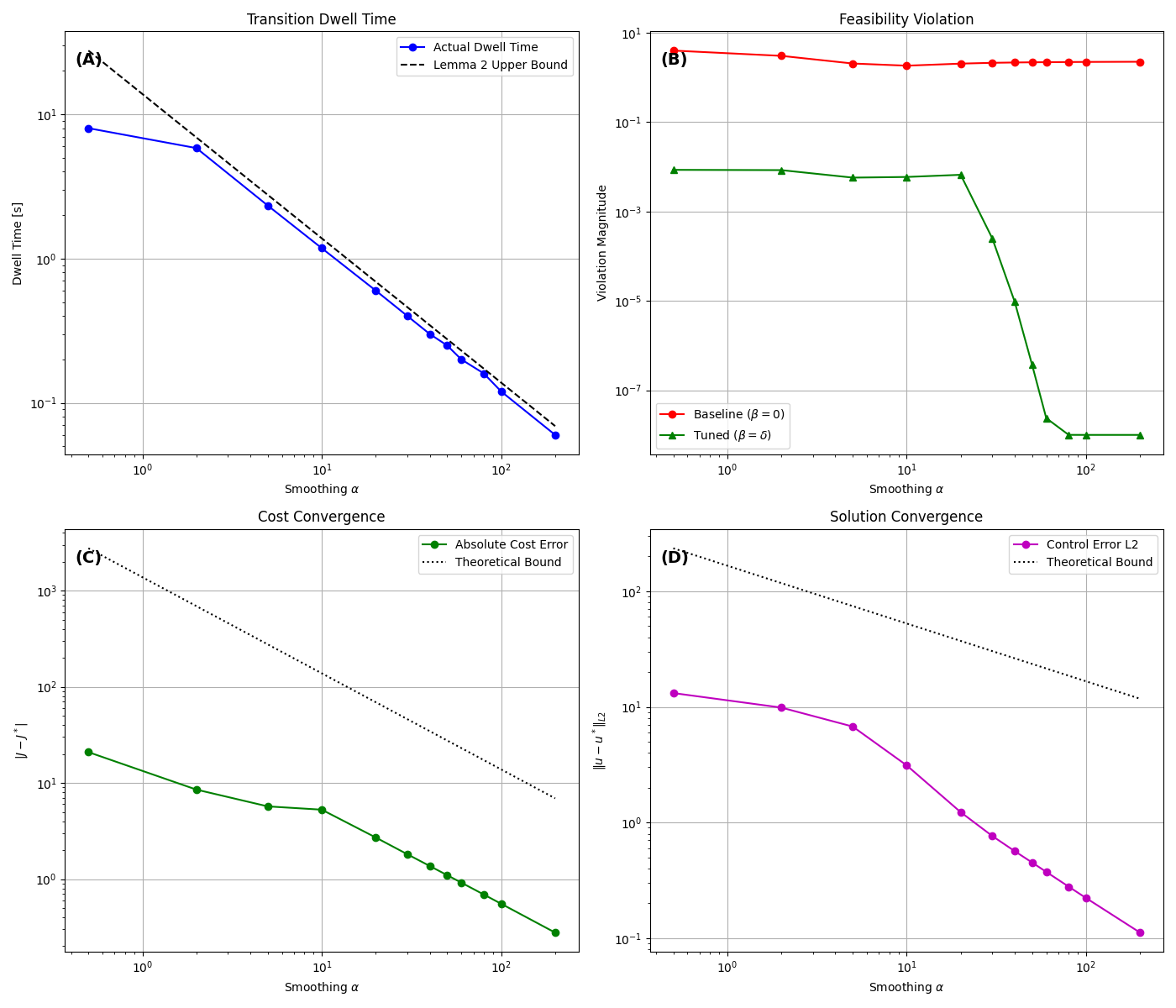}
\caption{Numerical verification of theoretical bounds. The four panels are arranged in a $2\times2$ grid: \textbf{(A, top-left)} transition dwell time decays as $\alpha$ increases; \textbf{(B, top-right)} feasibility violation is eliminated by the tuning rule $\beta=\delta$; \textbf{(C, bottom-left)} evaluation cost error converges to the global optimum; \textbf{(D, bottom-right)} control trajectory converges in $L_2$-norm. The panel labels (A)--(D) are printed inside each subplot.}
\label{fig:verification_suite}
\end{figure}

\section{Industrial Polymerization Process Control Case Study}
\label{sec:case}
\subsection{Process Description and Modeling}
The case study targets an operating gas--liquid polymerization train that cycles through several commercial grades on a 24~h schedule \cite{li2024benchmark}. Pronounced nonlinearities, tight thermal coupling, and mode-dependent safety limits make the plant a stringent NMPC benchmark.
\begin{figure}[H]
	\centering
	\includegraphics[width=0.8\columnwidth]{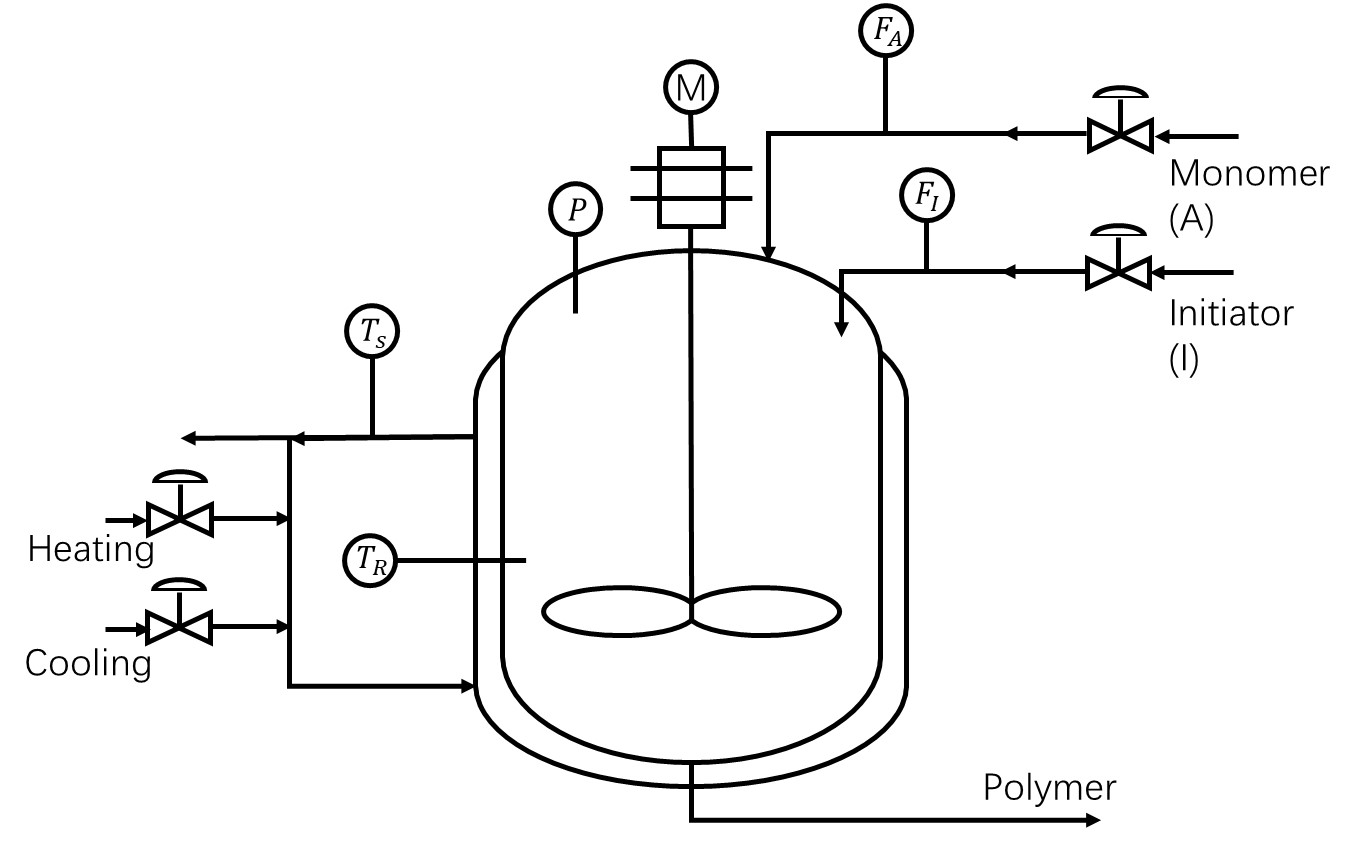}
	\caption{Industrial batch polymerization reactor with gas-liquid two phase}
	\label{fig:process}
\end{figure}

Each batch traverses three distinct modes: start-up heating and pressurisation, an isothermal polymerisation hold, and a finishing depressurisation that preserves temperature. Online quality indicators trigger feed retuning at every batch. The core equipment is a stirred-tank reactor with recirculation and a split-range jacket loop that alternates hot and cold utilities to coordinate monomer and initiator dosing with thermal control.

Control design uses the mechanistic model of Li et al.~\cite{li2024benchmark}, which describes the evolution of polymer chain moments, reactor temperature, and pressure through coupled mass and energy balances. Detailed reaction mechanics, differential equations, and model parameters are provided in the Supplementary Material. Here, we summarize the control-relevant variables and operational phases.

The reactor state vector $x \in \mathbb{R}^{12}$ collects polymer moments ($c_{\lambda_0}, c_{I^*}, c_A, c_{I_2}, c_B$), temperatures ($T_{\rm r}, T_{\rm J}$), gas phase variables ($N_{\rm gA}, P$), and cumulative inventories ($M_{A_{in}}, M_{B_{in}}, V_{\rm l}$):
\begin{equation}
\label{eq:state_def}
x = [c_{\rm \lambda_0}, c_{\rm I^*}, c_{\rm A}, c_{\rm I_2}, c_{\rm B}, T_{\rm r}, T_{\rm J}, N_{\rm gA}, P, M_{A_{in}}, M_{B_{in}}, V_{\rm l}]^T
\end{equation}
Only five variables are measured online:
\begin{equation}
\label{eq:output_def}
y = [T_{\rm r}, T_{\rm J}, P, M_{A_{in}}, M_{B_{in}}]^T.
\end{equation}
The remaining states are reconstructed via an Extended Kalman Filter (Section~\ref{subsec:ekf}). The manipulated inputs $u \in \mathbb{R}^3$ are the monomer feed $F_{\rm A}$, initiator feed $F_{\rm I_2}$, and jacket split-range position $\psi$:
\begin{equation}
\label{eq:input_def}
u = [F_{\rm A}, F_{\rm I_2}, \psi]^T.
\end{equation}
These inputs are implemented as piecewise-constant trajectories matching the NMPC discretization.

\subsection{Description of the Control Problem}
The controller seeks the fastest batch that respects safety and product specifications by manipulating $F_{\rm A}$, $F_{\rm I_2}$, and $\alpha$. Operation is divided into three modes: start-up ramps $T_{\rm r}$ and $P$ to $T_{\rm r,sp}=351.15~\text{K}$ and $P_{\rm sp}=1.5~\text{MPa}$ while polymerization begins; the holding phase maintains $T_{\rm r}$ within $\pm1~\text{K}$ and $P$ within $\pm0.1~\text{MPa}$ until $M_{\rm A,in}=\int_0^{t_f}F_{\rm A}\,\mathrm{d}t=3250~\text{kg}$ is delivered; and finishing stops the monomer feed, keeps temperature on setpoint, and relieves pressure to $P^{\rm end}=1~\text{MPa}$. The initiator charge $M_{\rm B,in}=0.002~\text{kg}$ can be distributed across any stage, adding a degree of freedom. Tables~\ref{tab:x} and \ref{tab:u} summarize state and input limits together with initial conditions.
Absolute pressures in Tables~\ref{tab:x} are retained in Pascals to match plant instrumentation, whereas the performance metrics reported in Section~\ref{sec:case} use MPa for clarity when discussing deviations.

\begin{table}[!hb]
	\centering
	\caption{Initial conditions and state constraints.}
	\label{tab:x}
	\begin{tabular}{lllll}
		\hline
		State                          & Init. cond. & Min.   & Max.    & Unit \\ \hline
		$(c_{\rm \lambda_0}V_{\rm l})$ & 0           & 0      & Inf     & mol  \\
		$(c_{\rm I^*} V_{\rm l})$      & 0           & 0      & Inf     & mol  \\
		$(c_{\rm A} V_{\rm l})$        & 914.2858    & 0      & Inf     & mol  \\
		$(c_{\rm I_2} V_{\rm l})$      & 0.0003632   & 0      & Inf     & mol  \\
		$(c_{\rm B} V_{\rm l})$        & 8.4         & 0      & Inf     & mol  \\
		$T_{\rm r}$                    & 340         & 320.15 & 373.15  & K    \\
		$T_{\rm J}$                    & 340         & 283.15 & 373.15  & K    \\
		$N_{\rm gA}$                   & 113.2038    & 0      & 10000   & mol  \\
		$P$                            & 160000      & 150000 & 1600000 & Pa   \\
		$V_{\rm l}$                    & 4000        & 2000   & 6000    & L    \\ \hline
	\end{tabular}
\end{table}

\begin{table}[!h]
	\centering
	\caption{Bounds on the manipulated variables.	}
	\label{tab:u}
	\begin{tabular}{llll}
		\hline
		Control       & Min. & Max.  & Unit \\ \hline
		$F_{\rm A}$   & 0    & 0.694 & kg/s \\
		$F_{\rm I_2}$ & 0    & 4$\times10^{-7}$ & kg/s \\
		$\psi$      & -1   & 1     & -    \\ \hline
	\end{tabular}
\end{table}

Measurements and associated sensor noise are summarized in Table~\ref{tab:y}.
\begin{table}[!h]
	\centering
	\caption{Measurement variables and measurement noise}
	\label{tab:y}
	\begin{tabular}{llll}
		\hline
		Variable & Maximum error & Unit & Noise distribution \\
		\hline
		$T_{\rm r}$    & 0.1           & K    & Gaussian           \\
		$T_{\rm J}$    & 0.1           & K    & Gaussian           \\
		$P$            & 5000          & Pa   & Gaussian           \\
		$M_{\rm A,in}$ & 0.1           & kg   & Gaussian           \\
		$M_{\rm B,in}$ & 1.e-09      & kg   & Gaussian           \\
		\hline
	\end{tabular}
\end{table}

\subsection{Cost Functions and Constraints}
As described in the previous subsection and illustrated in Figure~\ref{fig:mode_sketch}, the batch advances through heating/pressurizing, holding, and finishing stages that are triggered by state-dependent switching rules. Figure~\ref{fig:mode_sketch} sketches the idealized evolution of the key process variables ($T_{\rm r}$ and $P$) alongside the associated phases and switching conditions.

\begin{figure}[H]
    \centering
    \includegraphics[width=0.85\columnwidth]{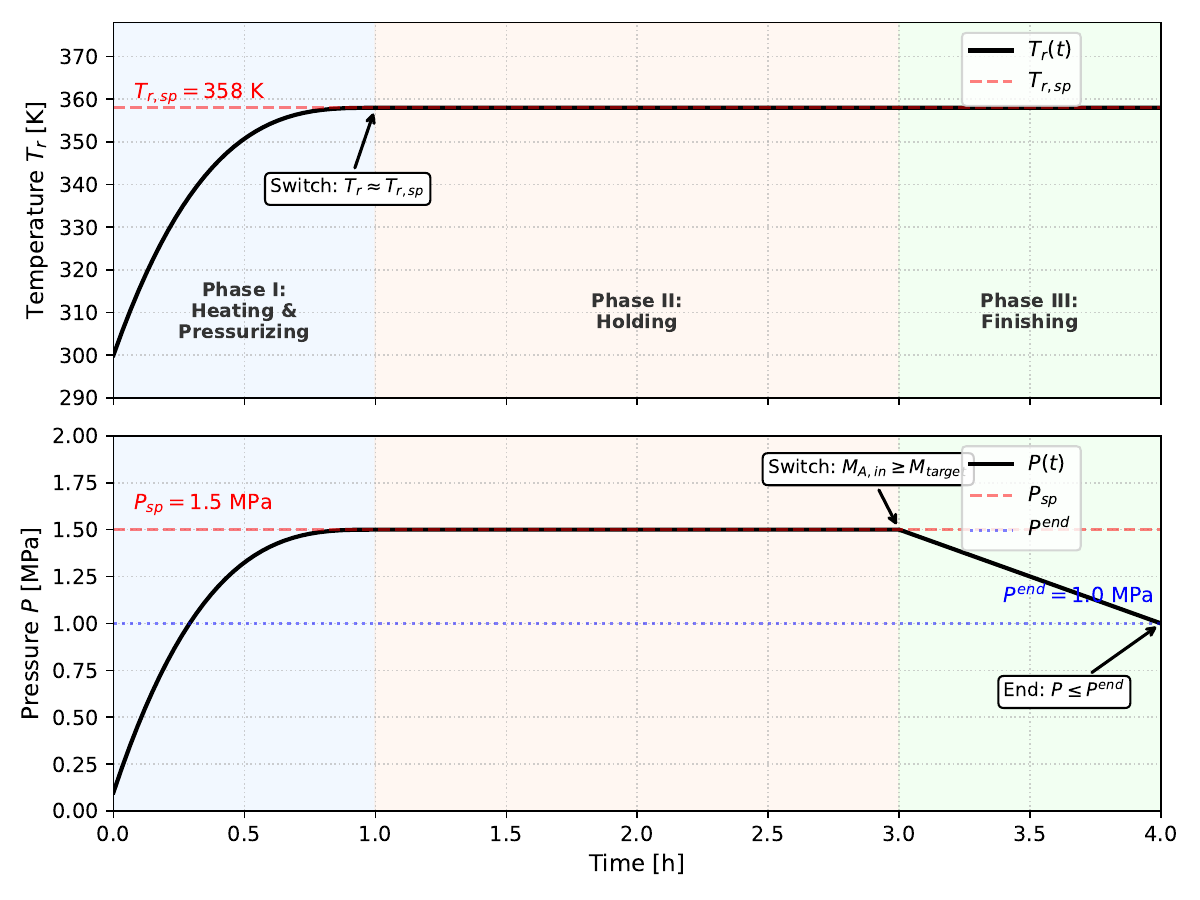}
    \caption{Schematic of the mode transition logic. The process cycles through three phases: (I) Heating and Pressurizing until setpoints ($T_{\rm r,sp}, P_{\rm sp}$) are reached; (II) Holding until the monomer dosage target is met; and (III) Finishing, where pressure drops to $P^{\rm end}$ while temperature is maintained.}
    \label{fig:mode_sketch}
\end{figure}

Each stage uses a dedicated running cost, integrated over its active horizon, together with mode-specific constraints.
The heating and pressurizing stage employs the running cost
\begin{equation}
\label{eq:cost1}
\begin{aligned}
    \ell_1(x(t),u_j,\Delta u_j) ={}& -R_{\rm A} + q_1(T_{\rm r} - T_{\rm r,sp})^2 + q_2(P - P_{\rm sp})^2 \\
    & {}+ r_1 (\Delta F_{\rm A,j})^2 + r_2 (\Delta F_{\rm I_2,j})^2 + r_3 (\Delta \psi_j)^2
\end{aligned}
\end{equation}
with stage contribution $\int_{t_{\rm h}^{\rm start}}^{t_{\rm h}^{\rm end}} \ell_1(x(t),u_j,\Delta u_j)\,{\rm d}t$. The first term maximizes monomer utilization, the quadratic tracking terms maintain the temperature and pressure set-points, and the increment penalties temper aggressive actuator movements. Here, index $j$ denotes the control interval within the prediction horizon. $\Delta F_{\rm A,j} = F_{\rm A,j} - F_{\rm A,j-1}$, $\Delta F_{\rm I_2,j} = F_{\rm I_2,j} - F_{\rm I_2,j-1}$, and $\Delta \psi_j = \psi_j - \psi_{j-1}$ (with the first increment taken as zero).
The monomer consumption rate is computed as
\begin{equation}
R_{\rm A} = \left(k_i  c_{\rm I^*} + (k_{\rm g} + k_{\rm tr,m})c_{\lambda_0 } \right)c_{\rm A}V_{\rm l}.
\end{equation}
Here $R_{\rm A}$ is the instantaneous monomer consumption rate (mol/s); maximising $R_{\rm A}$ indirectly minimises batch duration, addressing the ``minimize heating time'' objective in Table~\ref{tab:poly_modes}.

\paragraph{Combinatorial mode complexity} Although only three operational modes are defined (Table~\ref{tab:poly_modes}), the control horizon may span multiple mode entries and exits. Consider the polymerization phase: if temperature oscillates near the switching boundary $T_{\rm r} = T_{\rm r,sp}$, the predicted trajectory may repeatedly cross between phases, entering and leaving the same mode multiple times. Over a 60-step horizon, this generates a sequence of mode transitions that grows combinatorially with the number of crossings. Traditional mixed-integer formulations would require enumerating all such sequences, whereas the smoothing approach avoids this by continuously blending mode-dependent costs and constraints, allowing gradient-based solvers to find locally optimal solutions without explicit mode enumeration.
Stage-specific path constraints are collected in $g_1(x(t),u(t))$:
\begin{equation}
\label{eq:mpc1}
\begin{aligned}
    320.15 \leq T_{\rm r} \leq 373.15, \\
    320.15 \leq T_{\rm J} \leq 373.15, \\
    1.5\times10^5 \leq P \leq 1.6\times10^6, \\
    0 \leq M_{\rm B,in} \leq 0.002, \\
    0 \leq M_{\rm A,in} \leq 3250.
\end{aligned}
\end{equation}
Once $T_{\rm r}$ and $P$ reach their set-points, the controller transitions to the holding stage. The running cost remains $\ell_1$, whereas the admissible set narrows to reflect the tighter quality requirements:
\begin{equation}
\label{eq:mpc2}
\begin{aligned}
    T_{\rm r,sp} - 0.7 \leq T_{\rm r} \leq T_{\rm r,sp} + 0.7, \\
    320.15 \leq T_{\rm J} \leq 373.15, \\
    P_{\rm sp} - 0.1\times10^6 \leq P \leq P_{\rm sp} + 0.1\times10^6, \\
    0 \leq M_{\rm B,in} \leq 0.002, \\
    0 \leq M_{\rm A,in} \leq 3250.
\end{aligned}
\end{equation}
When the cumulative monomer feed reaches its target, the batch enters the finishing stage. The running cost then focuses on temperature regulation and actuator smoothness:
\begin{equation}
\label{eq:cost3}
\begin{aligned}
    \ell_3(x(t),u_j,\Delta u_j) ={}& q_1(T_{\rm r} - T_{\rm r,sp})^2 \
    & {}+ r_1 (\Delta F_{\rm A,j})^2 + r_2 (\Delta F_{\rm I_2,j})^2 + r_3 (\Delta \psi_j)^2
\end{aligned}
\end{equation}
with contribution $\int_{t_{\rm f}^{\rm start}}^{t_{\rm f}^{\rm end}} \ell_3(x(t),u_j,\Delta u_j)\,{\rm d}t$. The corresponding constraints $g_3(x(t),u(t))$ keep the process within safe operating limits:
\begin{equation}
\label{eq:mpc3}
\begin{aligned}
    T_{\rm r,sp} - 0.7 \leq T_{\rm r} \leq T_{\rm r,sp} + 0.7, \\
    320.15 \leq T_{\rm J} \leq 373.15, \\
    1.5\times10^5 \leq P \leq 1.6\times10^6, \\
    0 \leq M_{\rm B,in} \leq 0.002, \\
    0 \leq M_{\rm A,in} \leq 3250.
\end{aligned}
\end{equation}
\subsection{Implementation and Results}
This section presents an analysis of the performance of the proposed real-time NMPC framework in a specific scenario that arises in the operation of an industrial batch gas-liquid two-phase polymerization reactor.

\subsubsection{Implementation Details}
All results presented in this section are obtained in closed-loop simulation using the mechanistic benchmark model described in Supplementary Material as the plant surrogate. The NMPC problem is solved every 30~s with a prediction horizon of 60 steps, providing a 30~min look-ahead window. The finite-dimensional program is generated in CasADi \cite{andersson2019} using third-order direct collocation with four Radau points per interval \cite{vonstryk1993} and is handled by IPOPT \cite{biegler2009}. With 12 differential states and three inputs, the transcription contains 3\,072 decision variables ($3\times60$ input moves, $12\times60\times4$ collocation states, and 12 initial values), 2\,892 defect constraints, and 6\,144 bound constraints. The advanced-step warm start, variable scaling, and fixed iteration cap described in Section~\ref{sec:rt} are all active; together they keep the mean online computation at ${\sim}6$~ms while the background iteration requires 3.58~s (Table~\ref{tab:ablationResult}). To account for computational delay, the previously implemented control move is held constant until a fresh iterate is accepted. This delay mechanism is critical for the comparisons: scenarios with high computational cost (e.g., No-FIL) suffer from significant feedback delays, degrading performance despite having 'exact' solutions.

All simulations were executed on an Intel Core i7-10700F (2.90~GHz) with 8~GB of RAM, representative of the computing platforms available for industrial control \cite{vukov2015}. The resulting workload remains comfortably below the 30~s sampling budget, demonstrating that the proposed architecture can be realised with commodity hardware.

\subsubsection{PID Benchmark Configuration}
The single-loop PID benchmark replicates the specific supervisory logic currently deployed at the industrial plant modeled in this study. Since the process model is derived from an operating reactor, this baseline allows us to quantify performance improvements relative to the actual legacy system rather than a generic or idealized standard. Although cascade or multi-loop configurations can improve regulatory performance, our intent is to compare against real plant practice rather than an idealised baseline. This choice ensures that the reported improvements are directly relevant to practitioners considering NMPC upgrades.
The controller operates at a 1~s sampling interval and shares the same measurements and actuator limits as the NMPC benchmark.
\begin{itemize}[leftmargin=*]
	\item \textbf{Pressure-to-feed loop:} The primary loop regulates the reactor pressure to $P_{\rm sp}=1.5$~MPa by manipulating the monomer feed rate $F_{\rm A}$. A PI law with $K_{\rm p}=10$ and integral time $T_{\rm i}=500$~s (zero derivative action) is used, with the integral contribution clipped to respect $0 \leq F_{\rm A} \leq 0.694$~kg\,s$^{-1}$. The feed is kept at zero during the 900~s preparation stage and once the cumulative monomer charge exceeds 3\,250~kg.
	\item \textbf{Initiator dosing logic:} Initiator addition follows on/off rules. When $P<1.4$~MPa or the cumulative monomer inventory exceeds 2\,950~kg, the initiator feed $F_{\rm I_2}$ is held at zero. Otherwise, the controller switches between $8\times10^{-7}$ and $1\times10^{-7}$~kg\,s$^{-1}$ depending on whether the accumulated initiator mass has reached 1.2~g.
	\item \textbf{Thermal management:} Jacket temperature control uses a split-range structure. Below $T_{\rm r}=345$~K, the hot utility is selected ($T_{\rm jin}=363$~K) and a proportional loop with $K_{\rm p}=5$ adjusts the hot-water flow with saturation $0$--$1$~kg\,s$^{-1}$. Above 345~K, the coolant is selected ($T_{\rm jin}=280$~K) and a PI loop with $K_{\rm p}=20$, $K_{\rm i}=1\times10^{-3}$~s$^{-1}$ governs the cold-water flow in the range $0$--$30$~kg\,s$^{-1}$. The resulting flow commands are mapped to the normalized jacket-valve signal used in the process model.
\end{itemize}
These rules are executed together with the constraint handling and performance monitoring routines in the simulation scripts, yielding the PID trajectories reported in Table~\ref{tab:performance_comparison}.

\subsection{Extended Kalman Filter for State Estimation}
\label{subsec:ekf}
For the unmeasured states in our system, we employ an Extended Kalman Filter (EKF) \cite{ribeiro2004kalman} to estimate their values. The EKF is implemented with the following key parameters:

The measurement vector ${y}$ consists of five process variables:
\begin{equation*}
	{y}_k = \left[  T_{\rm r} ,
	T_{\rm J} ,
	P ,
	M_{\rm A,in} ,
	M_{\rm B,in}\right]^\top	
\end{equation*}

The measurement noise covariance matrix ${R}$ is defined by the standard deviations:
\begin{equation*}
	{R} = \text{diag}\left(\left[\frac{0.1}{3}, \frac{5000}{3}, \frac{0.1}{3}, \frac{0.1}{3}, \frac{1 \times 10^{-6}}{3}\right]^2\right)
\end{equation*}

The process noise covariance matrix ${Q}$ is set to:
	${Q} = \text{diag}([1\times10^{-6}, 1\times10^{-6}, 1, (3\times10^{-5})^2, 0.01, 0.1, 0.04, $
${5000}/{3}, 0.01, 1\times10^{-6}, 1\times10^{-6}, 1])$

The initial state covariance matrix ${P}_{0|0}$ is initialized as:
	${P}_{0|0} = \text{diag}([1\times10^{-6}, 1\times10^{-6}, 1, (3\times10^{-5})^2, 0.01, 0.1, 0.04, {5000}/{3},$ $0.01, 1\times10^{-6}, 1\times10^{-6}, 1])$

The initial states ${x}_0$ are taken from Table \ref{tab:x}, which provides the nominal values for the system initialization.

%
%
%
\subsection{Ablation Studies}
To evaluate the efficacy of key components in our real-time NMPC framework for batch polymerization processes, we conducted a series of ablation experiments. These studies systematically removed or modified individual components of the framework while maintaining others constant, allowing us to quantify their specific contributions to overall controller performance.
We focused on four critical components of our framework:
\begin{enumerate}
	\item Smoothing strategy 
	\item Variable scaling strategy 
	\item Fixed Iteration Limit strategy 
\end{enumerate}

\begin{table}[]
	\caption{Ablation Scenarios.}
	\centering
	\label{tab:ablation}
\footnotesize
\setlength{\tabcolsep}{3pt}
	\begin{tabular}{lll}
		\hline
		\textbf{Scenario} &
		\textbf{Description}
		\\ \hline
		Real-time NMPC                   & Baseline                                                         \\
		No-asStep & Standard NMPC \\
		No-SMS &
		Without smoothing strategy \\
		No-VSS & Without variable scaling strategy     \\
		No-FIL & Without fixed iteration limit strategy  \\ \hline
	\end{tabular}
\end{table}

We evaluated the performance of each scenario using the following metrics:
\begin{enumerate}
	\item Computational Feasibility (CF): Percentage of control intervals where the solution is obtained within the sampling time.
	\item Economic Performance (EP): Calculated as 1 divided by the time required to complete the entire batch, it represents the production rate, normalized to the baseline scenario.
	\item Constraint Satisfaction (CS): Frequency and magnitude of constraint violations, expressed as:
	$\text{CS} = (1 - \sum |\text{constraint violations}|/\sum |\text{constraint limits}|) \times 100\% $
	\item Control Performance (CP): Integrated Absolute Error (IAE) and Max Error (ME) for key controlled variables (temperature), normalized to the baseline scenario.
\end{enumerate}

\begin{table}[]
	\centering
	\caption{Ablation Study Results.}
	\label{tab:ablationResult}
	\resizebox{1\columnwidth}{!}{%
	\begin{tabular}{llllll}
		\hline
		\textbf{Scenario} & \textbf{CF} & \textbf{EP} & \textbf{CS} & \multicolumn{2}{c}{\textbf{CP}}     \\ \cline{5-6} 
		&         &        &        & \textbf{IAE} & \textbf{ME} \\ \hline
		Real-time NMPC                     & \textbf{100.0\%}             & \textbf{100.0\%}             & \textbf{100.0\%}             & \textbf{100.0\%} & \textbf{100.0\%} \\
		No-asStep & 100.0\% & 99.3\% & 98.0\% & 105.8\%      & 191.9\%     \\
		No-SMS    & 100.0\% & 95.0\% & 73.3\% & 325.6\%      & 403.7\%     \\
		No-FIL    & 31.3\%  & 75.9\% & 0.0\%  & 10286.2\%    & 5075.9\%    \\
		No-VSS    & 87.8\%  & 76.0\% & 10.8\% & 11828.2\%    & 6756.1\%    \\ \hline
	\end{tabular}
 }
\end{table}

Table~\ref{tab:ablationResult} summarises how each architectural ingredient supports real-time operation. The baseline implementation keeps online iterations to 0.006~s on average (3.580~s in the background) while delivering best-in-class performance. Removing the advanced-step execution (No-asStep) leaves the closed-loop metrics essentially unchanged yet inflates the online computation time to 3.721~s, roughly $60\times$ slower and therefore incompatible with the sampling rate.

Excluding the smoothing strategy (No-SMS) preserves feasibility but erodes economics, constraint satisfaction, and control quality by $5$--$325\%$ and doubles the background effort. Eliminating the fixed iteration limit (No-FIL) or variable scaling (No-VSS) is catastrophic: feasibility collapses to $31.3\%$ and $87.8\%$, respectively, and control performance deteriorates by several orders of magnitude. These trends confirm that smoothing, scaling, and capped iterations are jointly responsible for safe, timely optimisation.

\begin{figure}[H]	
	\centering
	\begin{subfigure}[b]{0.46\columnwidth}
		\includegraphics[width=\columnwidth]{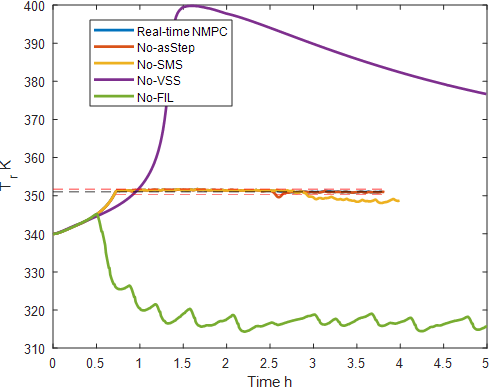}
		\caption{\normalsize$T_{\rm r}$}
		\label{fig:y1}
	\end{subfigure}
	\begin{subfigure}[b]{0.46\columnwidth}
		\includegraphics[width=\columnwidth]{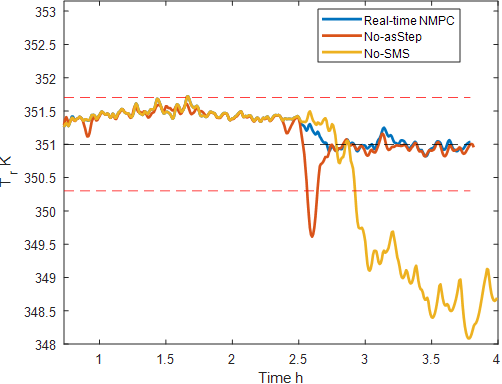}
		\caption{\normalsize$T_{\rm r}$}
		\label{fig:y2}
	\end{subfigure}
	\caption{Dynamic simulation}
	\label{fig:dy}
\end{figure}
Fig.~\ref{fig:dy} reports the temperature trajectories for the ablation scenarios. Panel~(a) shows that the baseline and No-SMS cases track the 351~K setpoint closely, whereas removing scaling (No-VSS) or iteration limits (No-FIL) generates severe constraint violations: the former overshoots the safety bound, the latter quenches the reactor. These behaviours align with the feasibility figures in Table~\ref{tab:ablationResult} (87.8\% and 31.3\%, respectively).

Panel~(b) zooms in on the well-behaved cases, highlighting how smoothing suppresses oscillations after the three-hour mark and how the advanced-step update dampens residual fluctuations. 

Variable scaling improves both conditioning and convergence speed. At $t=0$ a single primal-dual iteration takes 0.062~s with scaling and 0.073~s without. Figure~\ref{fig:convergence} compares the feasibility decay and confirms the faster contraction of both primal and dual residuals under the scaled formulation.
\begin{figure}[H]
	\centering
	\begin{subfigure}[b]{0.45\columnwidth}
		\includegraphics[width=\columnwidth]{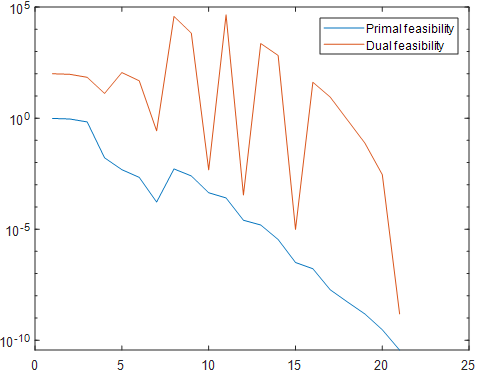}
		\caption{VSS}
	\end{subfigure}
	\begin{subfigure}[b]{0.45\columnwidth}
		\includegraphics[width=\columnwidth]{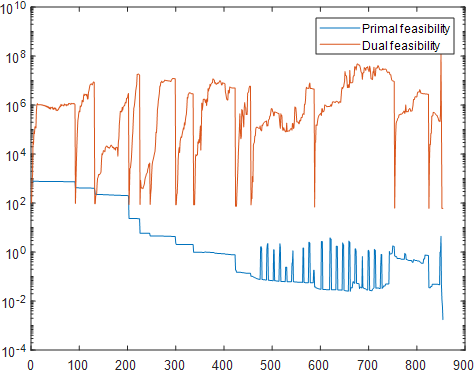}
		\caption{No-VSS}
	\end{subfigure}
	\caption{Convergence under VSS and No-VSS}
	\label{fig:convergence}
\end{figure}

These results collectively emphasize the synergistic nature of the proposed strategies in our real-time NMPC architecture. Each component plays a vital role in achieving robust, efficient, and safe control in demanding industrial polymerization processes. The smoothing strategy is crucial for maintaining performance, stability, and computational efficiency, while the fixed iteration limit and variable scaling strategies are essential for ensuring both the safety of the process and the real-time viability of the control algorithm.

\subsection{Robustness Tests and Comparison with PID Control}
The NMPC and PID configurations described above, including the EKF estimator in Section~\ref{subsec:ekf}, were evaluated under three operating scenarios to quantify robustness and benchmark performance against conventional practice:
\begin{enumerate}
    \item Nominal case without disturbances
    \item Worst-case disturbance scenario
    \item Random disturbances with 50 batch runs
\end{enumerate}

\begin{table*}[ht]
    \centering
    \caption{Performance comparison between NMPC and PID under different scenarios. IAE$_T$ and IAE$_P$ are computed as sample-wise sums of absolute errors; pressure deviations are normalised by $1.01\times10^5$~Pa ($\approx 0.1$~MPa) before accumulation.}
    \label{tab:performance_comparison}
    \small
    \resizebox{\textwidth}{!}{
    \begin{tabular}{llccccccc}
        \hline
        \textbf{Controller} & \textbf{Scenario} & \textbf{Time (h)}  & \textbf{$T_{\text{mean}}$} & \textbf{$T_{\text{max}}$} & \textbf{$P_{\text{mean}}$} & \textbf{$P_{\text{max}}$} & \textbf{IAE$_T$} & \textbf{CS (\%)} \\
        & & & (K) & (K) & (MPa) & (MPa) & (K$\cdot$sample) & \\
        \hline
        NMPC & No Disturbance & 3.81 & 0.30 & 0.73 & 0.147 & 1.21 & 111.68 & 99.56 \\
        NMPC & Random Dist. & 3.81 & 0.30 & 0.67 & 0.152 & 1.25 & 110.48 & 99.19 \\
        NMPC & Worst-case & 3.89 & 0.39 & 0.85 & 0.212 & 1.24 & 148.06 & 89.72 \\
        \hline
        PID & No Disturbance & 4.07 & 0.60 & 1.63 & 0.004 & 0.06 & 6811.39 & 71.68 \\
        PID & Random Dist. & 4.07 & 0.62 & 1.66 & 0.004 & 0.06 & 6970.20 & 71.19 \\
        PID & Worst-case & 5.56 & 5.23 & 51.13 & 0.095 & 3.44 & 87808.02 & 16.50 \\
        \hline
    \end{tabular}
    }
\end{table*}

Table~\ref{tab:performance_comparison} presents a comprehensive comparison between our NMPC framework and conventional PID control. Several key observations can be made:

\begin{enumerate}
    \item \textbf{Batch Time}: The NMPC controller consistently achieves shorter batch times (3.81-3.89 hours) compared to PID control (4.07-5.56 hours), representing a 6-31\% improvement in production efficiency.
    
    \item \textbf{Temperature Control}: NMPC demonstrates superior temperature control with mean deviations ($T_{\text{mean}}$) of 0.30-0.39K, significantly lower than PID's 0.60-5.23K. This is particularly evident in the worst-case scenario, where PID's maximum temperature deviation reaches 51.13K compared to NMPC's 0.85K.
    
    \item \textbf{Pressure Control}: While PID shows smaller pressure deviations in nominal conditions, NMPC maintains better pressure control under disturbances, particularly in the worst-case scenario where PID's maximum pressure deviation (3.44 MPa) is nearly triple that of NMPC (1.24 MPa).
    
    \item \textbf{Product Quality}: The NMPC controller maintains more consistent product quality metrics ($M_n$ and PDI) across all scenarios, while PID shows significant variations, particularly under worst-case disturbances.
    
    \item \textbf{Constraint Satisfaction}: NMPC maintains high constraint satisfaction rates (89.72-99.56\%) across all scenarios, while PID's performance degrades significantly under disturbances (71.68\% in nominal case to 16.50\% in worst-case).
\end{enumerate}

    
Fig.~\ref{fig:dynamic_comparison_nominal} depicts the nominal closed-loop trajectories. Both controllers stabilise the reactor, yet NMPC achieves tighter temperature and pressure regulation and a shorter batch by coordinating the monomer feed more aggressively while respecting constraints. 
\begin{figure}[H]
	\centering
	\includegraphics[width=\columnwidth]{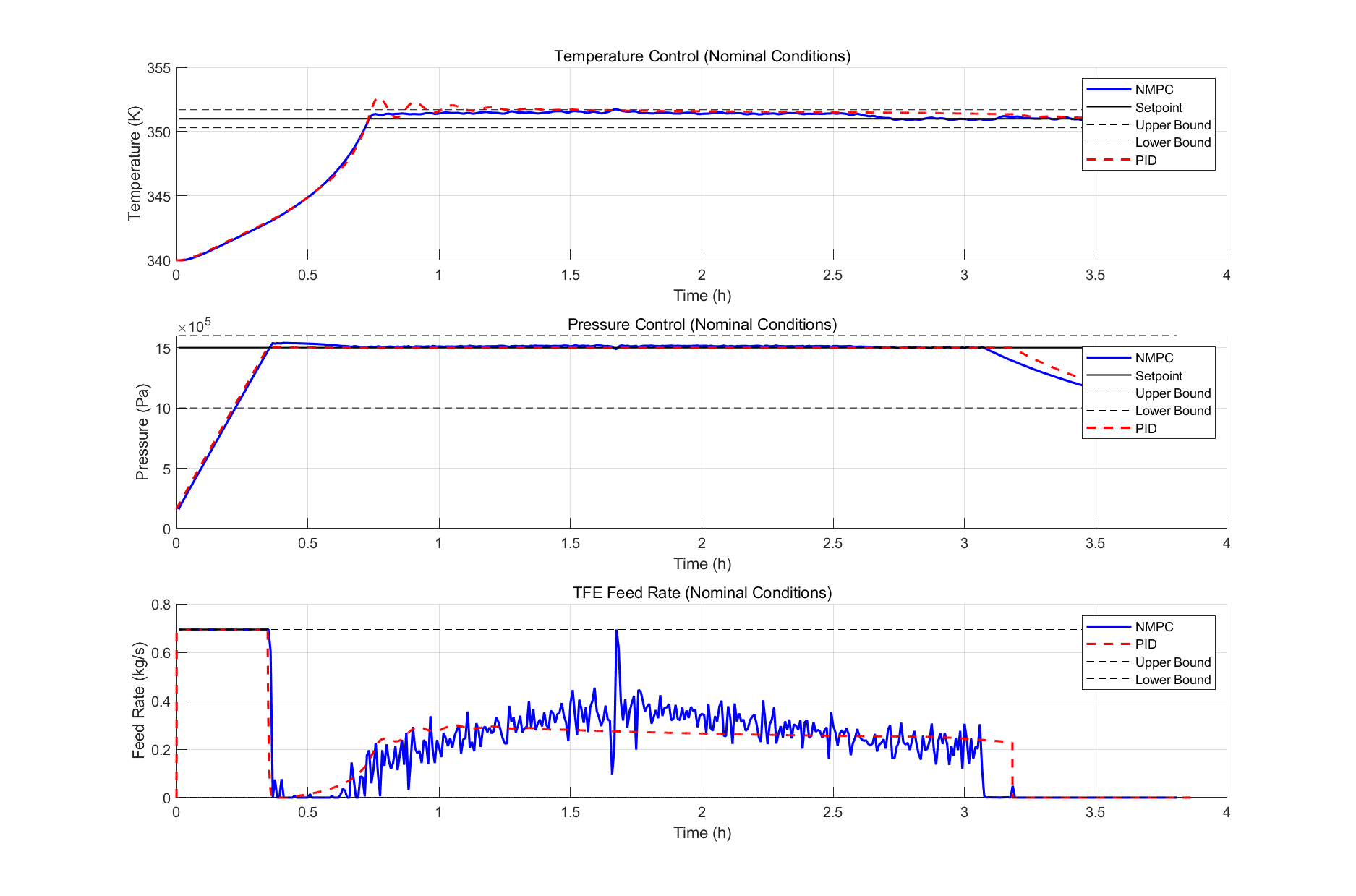}
	\caption{Dynamic response comparison between NMPC and PID under nominal scenario}
	\label{fig:dynamic_comparison_nominal}
\end{figure}

\begin{figure}[!ht]
	\centering
	\includegraphics[width=\columnwidth]{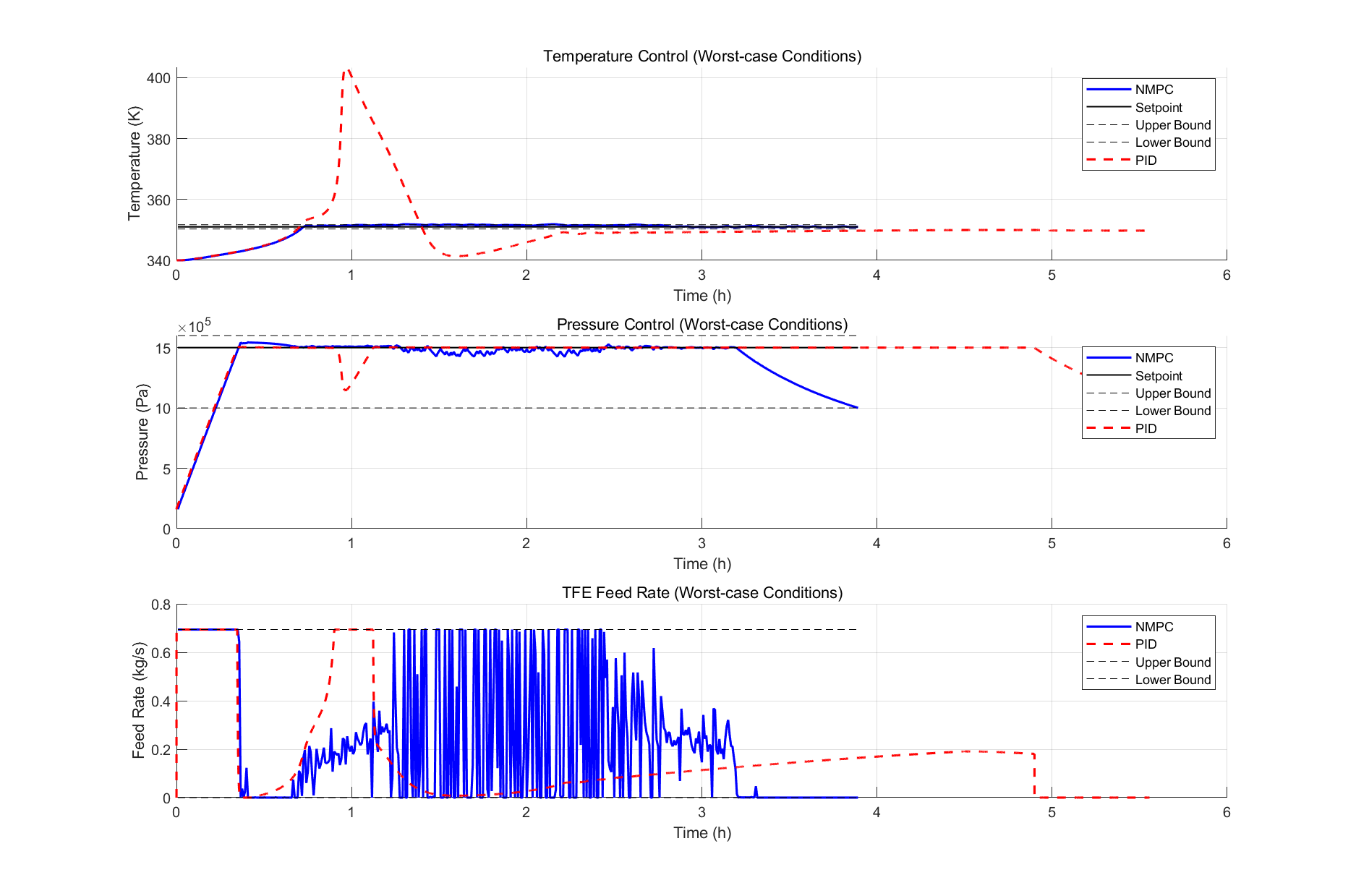}
	\caption{Dynamic response comparison between NMPC and PID under worst-case scenario}
	\label{fig:dynamic_comparison_worst}
\end{figure}

Figure~\ref{fig:dynamic_comparison_worst} shows that worst-case disturbances amplify these contrasts: PID oscillations lead to large temperature excursions (up to 51~K deviation) and delayed completion, consistent with Table~\ref{tab:performance_comparison}.

It is worth noting that the manipulated inputs in Fig.~\ref{fig:dynamic_comparison_worst} exhibit high-frequency oscillations during the temperature-holding phase. This phenomenon, known as \textit{active set chattering}, arises from the significant plant-model mismatch in this worst-case scenario. The controller attempts to operate the reactor precisely at the active thermal constraint boundary ($T_{r} \le T_{\rm r,sp} + 0.7$); however, prediction errors cause the system to repeatedly engage and disengage this constraint between sampling instances. Although this leads to aggressive control actions, the NMPC successfully maintains closed-loop stability and prevents thermal runaway, demonstrating robustness even under extreme parameter uncertainty.

Random-disturbance tests over 50 batches further confirm NMPC robustness. Temperature metrics remain within 1.1\% of nominal values, whereas PID performance degrades markedly as perturbations accumulate.

\section{Conclusions and Discussion}
\label{sec:con}
This paper presents an integrated real-time NMPC framework designed to address the computational challenges in industrial batch polymerization. By systematically combining smoothed switching, advanced-step warm-starting, variable scaling, and rigorous tuning, our approach bridges the gap between theoretical NMPC and industrial practice. This work delivers four key contributions:

\textbf{Theoretical Analysis:} We prove the smoothed switching formulation acts as an $\epsilon$-suboptimal approximation scheme. Explicit bounds (Theorem~\ref{thm:approximation_error}) and ``Feasibility Consistency'' ensure the numerical problem faithfully represents the original constraints without phantom infeasibility.

\textbf{Practical Tuning:} A systematic procedure balances approximation accuracy against numerical stability, resolving the trade-off between gradient explosion and model degradation.

\textbf{Integrated Real-Time Framework:} Fixed iteration limits and IPOPT scaling mechanisms enable these components to function cohesively under strict timing constraints, ensuring robust real-time performance.

\textbf{Validation:} Comparative analysis confirms superior disturbance rejection over industry-standard PID. Ablation studies quantify that neglecting delay compensation or smoothing significantly compromises stability and economic optimality.

The architecture developed in this research has broad applicability beyond polymerization processes. Its framework for managing computational constraints while maintaining control performance can be adapted to a wide range of nonlinear control problems across various industries, including chemical processing, energy systems, and advanced manufacturing.

For industrial applications, it is crucial to consider the time delays associated with data transmission. In this context, deploying NMPC on edge computing devices, particularly using dedicated hardware such as Field-Programmable Gate Arrays (FPGAs), presents an interesting opportunity. This approach can not only accelerate computations but also significantly reduce transmission latencies, further enhancing the real-time capabilities of our NMPC architecture.

It is worth noting that our simulations were conducted in an environment with relatively abundant computing resources. In actual industrial settings, multiple complex computations may need to be executed in parallel on the same computing platform. This reality opens an avenue: investigating methods for allocating and scheduling control-oriented computational resources in environments with limited or competitive computing capacity.

While our current work focuses primarily on addressing computational delays in NMPC deployment, we acknowledge that other challenges, such as model parameter uncertainty, remain to be fully explored. These factors can lead to more complex control problems with increased computational demands. However, the framework proposed in this paper provides a solid foundation for addressing these additional complexities.

In conclusion, this work represents a significant step forward in making NMPC more practical and efficient for industrial applications. By addressing the critical issue of computational delay, our approach brings the theoretical benefits of NMPC closer to widespread industrial adoption.
These advancements contribute to the development of process control methods and support their implementation in industrial applications, leading to improved product quality, increased energy efficiency, and enhanced operational safety in complex industrial processes.

\section*{Acknowledgments}
The authors declare that they have no known competing financial interests or personal relationships that could have appeared to influence the work reported in this paper. The postdoctoral fellowship of C.Z. was funded by KU Leuven through project ZKE5508.


\section*{Declaration of generative AI and AI-assisted technologies in the manuscript preparation process}
During the preparation of this work, the authors used ChatGPT (OpenAI, GPT-5) to assist with language editing and improve readability. After using this tool, the authors reviewed and edited the content as needed and take full responsibility for the content of the published article.

\section*{CRediT author statement}
\textbf{Chenchen Zhou}: Conceptualization, Methodology, Formal analysis, Software, Validation, Visualization, Writing â€“ original draft, Writing â€“ review and editing, Supervision, Project administration. \textbf{Zuzhen Ji}: Writing â€“ review and editing, Validation, Visualization, and Funding acquisition. \textbf{Jose Matias}: Supervision, Formal analysis, Writing â€“ review and editing, Resources, and Funding acquisition.

\bibliographystyle{elsarticle-num}
\bibliography{reference}

\appendix
\section{Proofs of the Results in Section~\ref{sec:theoretical_analysis}}
\label{app:proofs}
This appendix collects the proofs of Lemmas~\ref{lem:transition_region}--\ref{lem:cost_function_bound} and Theorem~\ref{thm:approximation_error}.

\subsection*{Proof of Lemma~\ref{lem:transition_region}}
When $h_i(x)\ge 0$ we have $p_i(x)=1$ and $|p_i^{\text{obj}}(x)-p_i(x)|>\epsilon$ is equivalent to
$1-\sigma(\alpha_i h_i(x))>\epsilon$. Using the identity $1 - \sigma(z) = \frac{e^{-z}}{1+e^{-z}}$ for $\sigma(z)=(1+\exp(-z))^{-1}$ with $z=\alpha_i h_i(x)$, this inequality implies:
\[
\frac{e^{-z}}{1+e^{-z}} > \epsilon \implies e^{-z} > \epsilon(1+e^{-z}) \implies e^{-z}(1-\epsilon) > \epsilon \implies e^{-z} > \frac{\epsilon}{1-\epsilon}.
\]
Taking natural logarithms yields $-\alpha_i h_i(x) > \ln(\frac{\epsilon}{1-\epsilon})$, or $h_i(x) < \frac{1}{\alpha_i}\ln(\frac{1-\epsilon}{\epsilon}) = \delta_i(\alpha_i)$. When $h_i(x)<0$ we have $p_i(x)=0$ and the condition
$|p_i^{\text{obj}}(x)-p_i(x)|>\epsilon$ reduces to $\sigma(\alpha_i h_i(x))>\epsilon$, namely
$\exp\bigl(-\alpha_i h_i(x)\bigr)<(1-\epsilon)/\epsilon$, which yields $h_i(x)>-\delta_i(\alpha_i)$. Combining the
two bounds shows that $|h_i(x)|<\delta_i(\alpha_i)$ is necessary and sufficient for the approximation
error to exceed $\epsilon$, and the stated result follows. \hfill$\square$

\subsection*{Proof of Lemma~\ref{lem:dwell_time}}
Refer to the transition region $\mathcal{R}_i(\alpha_i)$ characterized in Lemma~\ref{lem:transition_region} with half-width $\delta_i(\alpha_i)$. Consider a single crossing event $k \in \{1,\dots, K_i\}$ where the trajectory traverses this region. Let $t_{\text{in}}^k$ and $t_{\text{out}}^k$ denote the time instants when the trajectory enters and leaves $\mathcal{R}_i$, respectively. The total variation of the switching function across this traversal corresponds to the full width of the region:
\[
    \left| h_i(x(t_{\text{out}}^k)) - h_i(x(t_{\text{in}}^k)) \right| = |(\pm \delta_i) - (\mp \delta_i)| = 2\delta_i(\alpha_i).
\]
By the fundamental theorem of calculus,
\[
    h_i(x(t_{\text{out}}^k)) - h_i(x(t_{\text{in}}^k)) = \int_{t_{\text{in}}^k}^{t_{\text{out}}^k} \frac{d}{dt} h_i(x(\tau)) \, d\tau.
\]
Using the transversality condition from Assumption~\ref{ass:switching_surface}, we have $\left| \frac{d}{dt} h_i(x(t)) \right| = |\nabla h_i^\top f(x,u)| \ge \nu_i > 0$ inside the region (where $\delta_i \le \bar{\delta}_i$). Therefore
\[
    2\delta_i(\alpha_i) = \left| \int_{t_{\text{in}}^k}^{t_{\text{out}}^k} \frac{d}{dt} h_i(x(\tau)) \, d\tau \right| \ge \nu_i (t_{\text{out}}^k - t_{\text{in}}^k),
\]
so $\Delta t^k \le 2\delta_i(\alpha_i)/\nu_i$. Summing over at most $K_i$ crossings gives the stated bound. \hfill$\square$

\subsection*{Proof of Lemma~\ref{lem:feasible_inclusion}}
(\emph{Left inclusion.}) Fix $x\in\mathcal S_{\text{orig}}$ and an index $i$. If $h_i(x)\ge 0$, then $p_i(x)=1$ and $p_i(x)\,g_i(x,u)\le 0$ implies $g_i(x,u)\le 0$. Since $p_i^{\text{cons}}(x)\in(0,1]$, we have $p_i^{\text{cons}}(x)\,g_i(x,u)\le 0\le\tau_{\mathrm{feas}}$. If $h_i(x)<0$, then $p_i(x)=0$ and the inequality is trivial. Therefore $x\in\mathcal S_{\text{smooth}}^{\tau}$.

(\emph{Right inclusion.}) Let $x\in\mathcal S_{\text{smooth}}^{\tau}$ and fix $i$. If $h_i(x)\ge 0$, then by monotonicity of $\sigma$ we have $p_i^{\text{cons}}(x)\ge \bar{p}_i(\alpha_i,\beta_i)=\sigma(-\alpha_i\beta_i)$. From $p_i^{\text{cons}}(x)\,g_i(x,u)\le \tau_{\mathrm{feas}}$ it follows that $g_i(x,u)\le \tau_{\mathrm{feas}}/\bar{p}_i(\alpha_i,\beta_i)$. Multiplying by $p_i(x)=1$ gives the desired inequality in $\mathcal S_{\text{orig}}^{+}(\tau)$. If $h_i(x)<0$, then $p_i(x)=0$ and the inequality is trivial. Therefore $x\in\mathcal S_{\text{orig}}^{+}(\tau)$. \hfill$\square$

\subsection*{Proof of Lemma~\ref{lem:cost_function_bound}}
Let $x(t)$ denote the trajectory generated by $U$. Lemma~\ref{lem:transition_region} bounds the band where $p_i^{\text{obj}}\neq p_i$, while Lemma~\ref{lem:dwell_time} shows that the residence time inside that band is $O(\alpha_i^{-1})$. Bounding the difference in stage costs by $C_i=\sup\{\|l_i(x,u)\|:|h_i(x)|\le\delta_i(\alpha_i),\,u\in\mathbb{U}\}$ therefore yields
\begin{align*}
	|J^{\text{smooth}}(x(t_k),U) - J(x(t_k),U)|
	&\le \sum_{i=1}^{M} C_i\,K_i\,\frac{2}{\alpha_i{\nu}_i}\ln\frac{1-\epsilon}{\epsilon}
	= \sum_{i=1}^{M} \frac{c_i}{\alpha_i}\\
	&\le \Bigl(\sum_{i=1}^{M} c_i\Bigr)\alpha_{\min}^{-1}
	= \gamma(\alpha_{\min}^{-1}),
\end{align*}
which decreases monotonically with the smoothing gains. \hfill$\square$

\subsection*{Proof of Theorem~\ref{thm:approximation_error}}
Since $\mathcal{S}_{\text{orig}} \subseteq \mathcal{S}_{\text{smooth}}^{\tau}$, $U^\ast(x)$ is feasible for the smoothed problem. Optimality of $U^{\text{smooth}\ast}$ gives
\[
  J^{\text{smooth}}(x, U^{\text{smooth}\ast}) \ \le\ J^{\text{smooth}}(x, U^\ast).
\]
By Lemma~\ref{lem:cost_function_bound},
\[
\begin{aligned}
  J(x, U^{\text{smooth}\ast})
  &\le J^{\text{smooth}}(x, U^{\text{smooth}\ast}) + \gamma(\alpha_{\min}^{-1})\\
  &\le J^{\text{smooth}}(x, U^\ast) + \gamma(\alpha_{\min}^{-1})\\
  &\le J(x, U^\ast) + 2\,\gamma(\alpha_{\min}^{-1}).
\end{aligned}
\]
The local second-order growth around $U^\ast$ implies $J(x,U)-J(x,U^\ast)\ge \tfrac{\mu}{2}\|U-U^\ast\|^2$ for $\|U-U^\ast\|\le r$, so for $U=U^{\text{smooth}\ast}$ we obtain
\[
  \tfrac{\mu}{2}\,\|U^{\text{smooth}\ast}-U^\ast\|^2 \ \le\ 2\,\gamma(\alpha_{\min}^{-1}),
\]
which yields the bound. If needed, one restricts $\alpha_i$ so that the bound holds within the radius $r$ where the growth condition is valid. \hfill$\square$

\clearpage
\section*{Supplementary Material}
\setcounter{equation}{0}
\renewcommand{\theequation}{S\arabic{equation}}
\setcounter{table}{0}
\renewcommand{\thetable}{S\arabic{table}}
\section*{S1. Detailed Process Model}
This supplementary material details the reaction mechanism and mathematical model of the industrial batch polymerization process described in Section 5 of the main manuscript.

\subsection*{S1.1 Polymerization Mechanism}
The free radical polymerization mechanism involves initiator decomposition, chain initiation, growth, transfer, and termination, as summarized in Table~\ref{table:polymerization_mechanism}.

\begin{table}[h]
	\caption{Mechanism of free radical polymerization}
	\label{table:polymerization_mechanism}
	\centering
	\begin{tabular}{l|ll}
		\hline
		\textbf{Reaction types} & \textbf{Reaction equation} & \\
		\hline
		Initiator decomposition & $\rm I_2 \xrightarrow{\textit{k}_{\rm d}} 2I^*$ & $k_{\rm d}=k_{\rm d0}\exp\left({-{E_{\rm d}}/({\rm R}T_{\rm r})}\right)$\\
		Chain initiation & $\rm A + I^* \xrightarrow{\textit{k}_i} P_1$ &$k_{\rm i}=k_{\rm i0}\exp\left({-{E_{\rm i}}/({\rm R}T_{\rm r})}\right)$\\
		Chain growth & $\rm P_n + A \xrightarrow{\textit{k}_g} P_{n+1}$ &$k_{\rm g}=k_{\rm g0}\exp\left({-{E_{\rm g}}/({\rm R}T_{\rm r})}\right)$\\
		Chain transfer to monomer & $\rm P_n + A \xrightarrow{\textit{k}_{tr,m}} D_n + P_1$ &$k_{\rm tr,m}=k_{\rm tr,m0}\exp\left({-{E_{\rm tr,m}}/({\rm R}T_{\rm r})}\right)$ \\
		Chain transfer to chain transfer agent & $\rm P_n + B \xrightarrow{\textit{k}_{tr,a}} D_n + I^*$ &$k_{\rm tr,a}=k_{\rm tr,a0}\exp\left({-{E_{\rm tr,a}}/({\rm R}T_{\rm r})}\right)$\\
		Chain termination & $\rm P_n + P_m \xrightarrow{\textit{k}_t} D_{n+m}$ &$k_{\rm t}=k_{\rm t0}\exp\left({-{E_{\rm t}}/({\rm R}T_{\rm r})}\right)$\\
		\hline
	\end{tabular}
\end{table}

\subsection*{S1.2 Mathematical Model Equations}
The macroscopic model is derived using the method of moments. The moment balances for the polymer chains are:
\begin{align}
    \frac{{\rm d}(c_{\lambda_0}V_{\rm l})}{{\rm d}t} &= \bigl(k_i c_{\rm A} c_{\rm I^*} - k_{\rm tr,a} c_{\rm B}c_{\lambda_0 } - k_{\rm t} c_{\lambda_0 }^2\bigr)V_{\rm l} \label{eq:mom1}\\
    \frac{{\rm d}(c_{\rm I^*} V_{\rm l})}{{\rm d}t} &= \bigl(2f k_{\rm d} c_{\rm I_2} - k_i c_{\rm A} c_{\rm I^*} + k_{\rm tr,a} c_{\rm B}c_{\lambda_0 }\bigr)V_{\rm l} \label{eq:mom2}\\
    \frac{{\rm d}(c_{\rm A} V_{\rm l})}{{\rm d}t} &= \left(\frac{1000}{M_{g_A}}\frac{F_{\rm g2l}}{V_{\rm l}} - k_i c_{\rm A} c_{\rm I^*} - (k_{\rm g} + k_{\rm tr,m})c_{\lambda_0 } c_{\rm A}\right)V_{\rm l} \label{eq:mom3}\\
    \frac{{\rm d}(c_{\rm I_2} V_{\rm l})}{{\rm d}t} &= \left(\frac{1000}{M_{g_{I_2}}}\frac{F_{\rm I_2,in}}{V_{\rm l}} - 2fk_{\rm d} c_{\rm I_2}\right)V_{\rm l} \label{eq:mom4}\\
    \frac{{\rm d}(c_{\rm B} V_{\rm l})}{{\rm d}t} &= -k_{\rm tr,a} c_{\lambda_0 } c_{\rm B} V_{\rm l}. \label{eq:mom5}
\end{align}

The reactor thermal and inventory balances are:
\begin{align}
    \frac{{\rm d}T_{\rm r}}{{\rm d}t} &= \frac{ Q_{\rm feed}+Q_{\rm w}+Q_{\rm r}-Q_{\rm loss}}{M_{\rm s}C_{\rm p,s}+M_{\rm A}C_{\rm p,A}} \label{eq:bal1}\\
    \frac{{\rm d}T_{{\rm J}}}{{\rm d}t} &= \frac{1}{M_j c_p}\Bigl(K_p(\psi) - UA(T_{{\rm J}} - T_{\rm r})\Bigr) \label{eq:bal2}\\
    \frac{{\rm d}N_{\rm gA}}{{\rm d}t} &= \frac{1000}{M_{\rm g_A}}(F_{\rm A,in} - F_{\rm g2l}) \label{eq:bal3}\\
    \frac{{\rm d}P}{{\rm d}t} &= \frac{R}{V_g/1000}\left(T_{\rm r} \frac{{\rm d}N_{\rm gA}}{{\rm d}t} + N_{\rm gA} \frac{{\rm d}T_{\rm r}}{{\rm d}t}\right) \label{eq:bal4}\\
    \frac{{\rm d}M_{A_{\rm in}}}{{\rm d}t} &= F_{\rm A,in} \label{eq:bal5}\\
    \frac{{\rm d}M_{B_{\rm in}}}{{\rm d}t} &= F_{\rm I_2,in} \label{eq:bal6}\\
    \frac{{\rm d}V_{\rm l}}{{\rm d}t} &= 0.06667\, F_{\rm g2l}. \label{eq:bal7}
\end{align}

The auxiliary relations are:
\begin{equation}
	\label{eq:auxiliary}
\begin{aligned}
	&F_{\rm g2l} = k(PH-1000c_{\rm A})\\
	&Q_{{\rm feed}}=F_{{\rm in,A}}C_{{\rm p,A}}(T_{{\rm A}}-T_{{\rm r}})\\
	&Q_{{\rm w}}=US(T_{{\rm J}}-T_{{\rm r}})\\
	&Q_{{\rm r}}=\Delta Hk_{\rm g}c_{\rm A}c_{\lambda_{0}}V_{{\rm l}}\\
	&Q_{{\rm loss}}=b_1(T_{{\rm r}}-T_{{\rm amb}})^{b_2} \\
	&K_p(\psi) = \left\{ 
	\begin{aligned}
		&-\psi F_{\rm clod,max}C_p(T_{\rm cold}-T_{\rm J}) & &-1\leq \psi \leq 0 \\
		&\psi F_{\rm hot,max} C_p(T_{\rm hot}-T_{\rm J})  & &0 < \psi \leq1 
	\end{aligned}
	\right.
\end{aligned}
\end{equation}

\subsection*{S1.3 Model Parameters}
\begin{table}[h]
	\centering
	\caption{Parameters of the gas-liquid two-phase reactor model.}
	\label{tab:para}
	\begin{tabular}{llll}
		\hline
		{Parameter} & {Description}                           & {Value} & {Units} \\ \hline
		$R$               & Ideal gas constant                            & 8.314         & $\mathrm{J/(mol\cdot K)}$     \\
		$V$               & Reactor volume                                & 6000          & L             \\
		$k_{10}$          & Pre-exponential factor for   reaction 1       & 1.13E+17      & 1/s           \\
		$k_{20}$          & Pre-exponential factor for   reaction 2       & 3.62E+15      & -             \\
		$k_{30}$          & Pre-exponential factor for   reaction 3       & 5.49E+07      & -             \\
		$k_{40}$          & Pre-exponential factor for   reaction 4       & 9.96E+07      & -             \\
		$k_{50}$          & Pre-exponential factor for   reaction 5       & 3.15E+06      & -             \\
		$k_{70}$          & Pre-exponential factor for   reaction 7       & 3.38E+09      & -             \\
		$E_1$             & Activation energy for reaction   1            & 1.35E+05      & J/mol         \\
		$E_2$             & Activation energy for reaction   2            & 119715        & J/mol         \\
		$E_3$             & Activation energy for reaction   3            & 17413.76      & J/mol         \\
		$E_4$             & Activation energy for reaction   4            & 53020         & $\mathrm{J/mol}$         \\
		$E_5$             & Activation energy for reaction   5            & 20000         & J/mol         \\
		$E_7$             & Activation energy for reaction   7            & 13604.5       & J/mol         \\
		$k$               & Gas-liquid mass transfer rate   constant of A & 0.0562        & $\mathrm{kg\cdot m/(s\cdot mol)}$ \\ 
        $H$               & Henry's constant                                 & 1/700         & $\mathrm{mol/(Pa\cdot m^3)}$   \\
        $M_{g_A}$         & Relative molecular mass of A                     & 100           & g/mol         \\
        $M_{g_B}$         & Relative molecular mass of B                     & 228           & g/mol         \\
        $C_{p_A}$         & Specific heat capacity of A                      & 804           & $\mathrm{J/(kg\cdot K)}$      \\
        $T_A$             & Temperature of A feed                            & 333           & K             \\
        $a$               & Coefficient of heat loss to the   ambient        & 1             & -             \\
        $b$               & Coefficient of heat loss to the   ambient        & 2             & -             \\
        $T_{amb}$         & Ambient temperature                              & 293           & K             \\
        $Q_{stir}$        & Stirring heat                                    & 0.5           & W             \\
        $M$               & Mass of reactor liquid phase                     & 4000          & kg            \\
        $M_j$             & Mass of jacket water                             & 2000          & kg            \\
        $C_p$             & Specific heat capacity of reactor   liquid phase & 4200          & $\mathrm{J/(kg\cdot K)}$      \\
        $T_{r_{sp}}$      & Reactor temperature set point                    & 351           & K             \\
        $P_{sp}$          & Reactor gas pressure set   point                 & 1500000       & Pa            \\ \hline
	\end{tabular}
\end{table}

\end{document}